\documentclass[11pt,reqno]{amsart}
\usepackage{CJK,CJKnumb,CJKulem,times,dsfont,ifthen,mathrsfs,latexsym,amsfonts,color}
\usepackage{amsmath,amsthm,fontenc,amssymb,bm,graphicx,psfrag,listings,curves,extarrows,enumitem,slashed}

\usepackage[backref=page]{hyperref}
\usepackage{cite}
\usepackage{geometry}
\usepackage{indentfirst}
\usepackage{amssymb}
\makeatletter

\newcommand{\Rmnum}[1]{\expandafter\@slowromancap\romannumeral #1@}
\makeatother

\usepackage[showonlyrefs]{mathtools}  
\mathtoolsset{showonlyrefs=true}

\newtheorem{definition}{Definition}[section]
\newtheorem{theorem}{Theorem}[section]
\newtheorem{lemma}{Lemma}[section]

\newtheorem{proposition}{Proposition}[section]
\newtheorem{remark}{Remark}[section]

\newcommand{\al}{\alpha}

\newcommand{\si}{\sigma}

\newcommand{\N}{{\mathbb N}}

\newcommand{\Z}{{\mathbb Z}}
\newcommand{\C}{{\mathbb C}}

\newcommand{\R}{{\mathbb R}}
\newcommand{\bean}{\begin{eqnarray*}}
	\newcommand{\eean}{\end{eqnarray*}}

\newcommand{\sbr}[1]{\left(#1\right)}
\newcommand{\mbr}[1]{\left[#1\right]}
\newcommand{\lbr}[1]{\left\{#1\right\}}
\newcommand{\hbr}[1]{\left\langle  #1 \right\rangle}

\newcommand{\dist}[1]{{\rm dist} \sbr{#1} }

\newcommand{\dx}{  ~\mathrm{d} x}

\newcommand{\dz}{  ~\mathrm{d} z}

\newcommand{\nm}[1]{\Vert #1 \Vert}

\setitemize{itemindent=38pt,leftmargin=0pt,itemsep=-0.4ex,listparindent=26pt,partopsep=0pt,parsep=0.5ex,topsep=-0.25ex}

\numberwithin{equation}{section}

\begin{document}
\theoremstyle{plain}

\title[dynamics of NLS with partial harmonic oscillator]{Dynamics of focusing nonlinear Schr\"odinger equation with partial harmonic confinement in higher dimensions}
	\date{}

\author[T. Liu]{Tianhao Liu}
\address{Tianhao Liu, Institute of Applied Physics and Computational Mathematics and National Key Laboratory of Computational Physics, Beijing 100088, China}
\email{liuthmath@gmail.com}

\author[Z. Ma]{Zuyu Ma}
\address{Zuyu Ma, Graduate School of China Academy of Engineering Physics, Beijing 100088, China}
\email{mazuyu23@gscaep.ac.cn}

\author[Y. Song]{Yilin Song}
\address{Yilin Song,  The Graduate School of China Academy of Engineering Physics, Beijing, 100088,\  China}
\email{songyilin21@gscaep.ac.cn}

\author[J. Zheng]{Jiqiang Zheng}
\address{Jiqiang Zheng, Institute of Applied Physics and Computational Mathematics and National Key Laboratory of Computational Physics, Beijing 100088, China}
\email{zheng\_jiqiang@iapcm.ac.cn, zhengjiqiang@gmail.com}

	\maketitle
	
	\begin{abstract}
    We study the following focusing intercritical nonlinear Schr\"odinger equation  with   partial harmonic confinement:
   \begin{equation*}
		\begin{cases}
			 i\partial_t u+\Delta_{z}u-y^2 u =- |u|^{\alpha}u,\quad t\in \R,    \\ 
             u(0,z)=  u_0(z), \  z=(x,y)\in  \R^d\times \R,
		\end{cases} 
	\end{equation*}where    $d \geq 1 $  is an integer and    the exponent $\alpha$  satisfies
      \begin{equation}\label{assumption}
      	4/d< \alpha<    {4}/\sbr{d-1}, \,\,\, \text{if} ~~ d\geq 2;    \quad \quad  4/d< \alpha< + \infty,\,\,\, \text{if}  ~~ d=1. 
  \end{equation}  
    For this model, A. Ardia and R. Carles [Comm. Math. Sci. 19 (2021), 993-1032] established a sharp scattering result below the ground state threshold in dimensions $d \leq 4$ via the concentration-compactness and rigidity argument. However, their approach breaks down in higher dimensions due to the lack of smoothness in the nonlinearity. In this paper, we introduce a new strategy that removes this dimensional restriction and extend their results to  higher dimensions by circumventing the concentration-compactness principle. 	The main ingredients of our work are the interaction Morawetz-Dodson-Murphy estimates and an alternative variational characterization of the ground state threshold.
	\end{abstract}
	

	\section{Introduction}\label{Sect1}
 \subsection{Background and motivation}
Consider the Cauchy problem for    the   nonlinear  Schr\"odinger equation (NLS)   with   a partial harmonic confinement   
\begin{equation}\label{PHNLS}
		\begin{cases} \tag{PHNLS}
			 i\partial_t u+\Delta_{z}u-y^2 u =-|u|^{\alpha}u,\quad t\in \R,    \\ 
             u(0,z)=  u_0(z),  
		\end{cases} 
	\end{equation}
    where $ z=(x,y)\in  \R^d\times \R$ with  $d \in \N\setminus\lbr{0}$,  the complex-valued function $u=u(t,z)\in \C$ is the unknown wave function. The exponent $\alpha$ is chosen in the intercritical range
      \begin{equation}\label{assumption}
      	4/d< \alpha<    {4}/\sbr{d-1}, \,\,\, \text{if} ~~ d\geq 2;    \quad \quad  4/d< \alpha< + \infty,\,\,\, \text{if}  ~~ d=1. 
  \end{equation}  
Here, the left endpoint of \eqref{assumption} corresponds to the mass-critical exponent in $d$    dimensions (the free $x$-direction), while the right  endpoint corresponds to the energy-critical exponent in $d+1$   dimensions (the full space). 
\vskip 0.08in
 The main purpose of this paper is to study the long-time behaviors of the solutions to \eqref{PHNLS}.
In particular, we     establish   a  complete dichotomy classification of scattering/blow-up    in   dimensions $d\geq 3$. The precise statement of our main result is given in Section \ref{Section:mainresult}. We now provide the background and motivation for studying the dynamics of \eqref{PHNLS}.

  \vskip0.08in
   Equation \eqref{PHNLS} is a special case of the   general  NLS with harmonic potential  
\begin{equation}\label{HNLS}
    \begin{cases} 
    i\partial_tu+\Delta_z u-(\kappa_1 |x|^2+\kappa_2  |y|^2) u=\mu|u|^{\alpha}u,  \quad t\in\R,\\
  u(0,z)=  u_0(z), 
\end{cases} 
\end{equation} 
where  $\kappa_1,\kappa_2\in \lbr{0,1}$,  $z=(x,y)\in  \R^d\times \R^N$ with $d,N\geq 1$, and we denote the total spatial dimension by $m = d + N$.  The parameter $\mu\in \lbr{+1,-1}$  with $\mu = +1$ corresponding to the defocusing case and $\mu = -1$ to the focusing case.  
 Equation \eqref{HNLS} is a fundamental model in quantum mechanics, which serves as a model for Bose–Einstein condensates confined in laboratory traps, as well as for describing the envelope dynamics of  a general dispersive wave in a weakly nonlinear medium.  The harmonic potential in  \eqref{HNLS} corresponds to the standard modelling for magnetic traps in the context of Bose–Einstein condensation; see, e.g., \cite{Josserand-Pomeau2001, Pitaevskii-Stringari2003}. Further physical motivation and background can be found in \cite{Bellazzini=CMP2017,Antonelli,Cheng-APDE}.

\vskip 0.08in

For the   autonomous case $\kappa_1 =\kappa_2=0$,  equation \eqref{HNLS} reduces to the standard NLS 
\begin{equation}\label{Free-NLS}
    \begin{cases}
    i\partial_tu+\Delta_z u=\mu|u|^{\alpha}u,&(t,z)\in\R\times\R^{m},\\
  u(0,z)=  u_0(z).
\end{cases} 
\end{equation} 
For the defocusing case, the global well-posedness and scattering in $H^1(\mathbb{R}^m)$ for \eqref{Free-NLS} are classical results, we refer to the Cazenave's textbook \cite{Caze}. For the focusing case, the global well-posedness was considered by Weinstein \cite{Weinstein=CMP}.
     The scattering was later proved by Holmer-Roudenko \cite{Holmer-CMP} for radial initial data and Duyckaerts-Holmer-Roudenko \cite{Duykaerts-MRL} for non-radial data with $(\alpha,m)=(2,3)$, and ultimately extended to any dimension and  any inter-critical power by Akahori-Nawa \cite{akahori}, Fang-Xie-Cazenave \cite{fang} and Guevara \cite{guevara}.  Their method is the application of concentration-compactness argument initially developed in Kenig-Merle \cite{KM-Invent} which was the first used to treat the energy-critical case. Later, Dodson-Murphy \cite{DM-Proc,DM-MRL} developed a new method to study the scattering for this model without relying on the concentration-compactness argument.

\vskip0.08in
 For the non-autonomous case where at least one $\kappa_j$ is nonzero,  the presence of   harmonic potential influences strongly the dynamics of the solution.  When all $\kappa_1=\kappa_2 = 1$, the harmonic potential reduces to the (isotropic) quadratic form $|z|^2$, leading to the following model 
\begin{equation} \label{QNLS}
    \begin{cases} 
    i\partial_tu+\Delta_z u-|z|^2 u=\mu|u|^{\alpha}u,&(t,z)\in\R\times\R^{m},\\
  u(0,z)=  u_0(z).
\end{cases} 
\end{equation}
Global well-posedness for \eqref{QNLS} has been established in a series of works: see \cite{Carles-MMMAS,Carles-AHP,Zhang} for the energy-subcritical case, and \cite{Killip-Visan-Zhang,Jao-CPDE,Jao-DCDS} for the energy-critical case.
It is well-known that  the solutions of \eqref{QNLS} do not scatter  because  the operator $-\Delta_z+|z|^2$  has a purely discrete spectrum.  We would also like to mention that    the long-time behavior of \eqref{QNLS} can be compared to the case of the pure torus $\mathbb{T}^m$. Indeed,   for the cubic NLS with $\alpha=2$ on $\mathbb{T}^2$, a non-scattering solution was constructed in \cite{CKSTT2010}. This behavior was later quantified by Guardia and Kaloshin \cite{Guardia}. In a similar spirit, Chabert \cite{Chabert} constructed a solution for the cubic NLS with a harmonic oscillator potential whose $H^1$ norm grows.

\vskip 0.04in
Nevertheless, if the confinement is turned off in some  but not all  directions, the condensate may still evolve asymptotically freely. This motivates the study of NLS \eqref{HNLS} with (anisotropic) partial harmonic potential. 
%
   Without loss of generality, we consider the case $\kappa_1=0$ and $ \kappa_2=1$,  which leads to the following model\begin{equation}\label{anisoNLS}
    \begin{cases} 
    i\partial_tu+\Delta_z u-|y|^2 u=\mu|u|^{\alpha}u,\quad t\in \R, \\
  u(0,z)=  u_0(z),
\end{cases} 
\end{equation}
where $ z=(x,y)\in  \R^d\times \R^N$. 
 Observe that  since  the potential is added in $y$-direction,   a global-in-time dispersive estimate is not expected for this direction. This can be seen from Mehler's formula 
\begin{align*}
	e^{it(\Delta_y-|y|^2)}f(x,y)=\frac{1}{(2\pi i\sin(2t))^\frac{N}{2}}\int_{\R^d}e^{\frac{i}{\sin(2t)} \big(\frac{y^2+\zeta^2}{2}\cos(2t)- y\cdot \zeta\big)}f(x,\zeta)d\zeta.
\end{align*}
Consequently, one obtains  following periodic-in-time dispersive estimate 
\begin{equation}
    	\big\|e^{it( \Delta_{y}-|y|^2)}f(x,y)\big\|_{L_{y}^\infty(\R^N)}\lesssim|\sin(2t)|^{-\frac{N}{2}}\|f(x,y)\|_{L_{y}^1(\R^N)},\quad\forall t\notin\frac{\pi}{2}\Bbb{Z}.
\end{equation}
However, by the strong dispersion of the $x$-direction, we have the global-in-time dispersive estimate 
\begin{align*}
	\big\|e^{it( \Delta_{x}+\Delta_{y}-|y|^2)}f(x,y)\big\|_{L_{x}^\infty L_{y}^2(\R^{d }\times\R^N)}\lesssim|t|^{-\frac{d}{2}}\|f(x,y)\|_{L_{x}^1L_{y}^2(\R^{d}\times\R^N)}.
\end{align*}
Therefore, one can expect the  scattering for equation \eqref{anisoNLS} in the intercritical case
\begin{equation} \label{s1}
    \frac{4}{d}< \alpha< \frac{4}{d+N-2},
\end{equation}
    which holds if and only if $N = 1$. It is therefore meaningful to investigate the long-time dynamics of the  solutions to  equation \eqref{anisoNLS} with one-dimensional partial harmonic confinement 
 \begin{equation} \label{equ1}
		\begin{cases}  
			 i\partial_t u+\Delta_{z}u-y^2 u =\mu|u|^{\alpha}u,\quad t\in \R,    \\ 
             u(0,z)=  u_0(z),  
		\end{cases} 
	\end{equation}
 under the assumption  \eqref{assumption}, where  $z=(x,y)\in  \R^d\times \R$.

   \vskip 0.12in
 Over the past decade, the long-time behavior of \eqref{equ1} has been studied by several authors.
 The pioneering work by Antonelli, Carles, and Silva \cite{Antonelli} first addressed this problem in the defocusing case $\mu=1$.   By establishing a new interaction Morawetz estimate associated with the partial harmonic oscillator, they proved scattering phenomenon under the assumption \eqref{assumption} for dimensions $1 \leq d \leq 3$  within the fully weighted Sobolev space
  $$  \lbr{ u \in H^1(\R^{d+1}) : |z|u \in L^2(\R^{d+1})}.$$ 
 Further results on the defocusing case can be found in    \cite{Cheng-APDE,Deng-Su-Zheng,Carles-Gallo-JMP}.

\vskip0.08in

  In this paper, we are particularly interested in the long-time behavior of solutions to \eqref{equ1} in the focusing case, that is, for equation \eqref{PHNLS}. Compared to the defocusing case, the focusing case is more delicate and requires a more subtle analysis.   In the study of \eqref{PHNLS}, the following mass and energy quantities are conserved 
  	\begin{equation}
 	    \begin{aligned}
 	        &M(u(t))=\int_{\R^{d+1}} |u(t,z)|^2 \dz,	\\
    &E(u(t))=\frac{1}{2} \int_{\R^{d+1}} |\nabla_z u(t,z)|^2 \dz+\frac{1}{2} \int_{\R^{d+1}} y^2|u(t,z)|^2 \dz -\frac{1}{\alpha+2}  \int_{\R^{d+1}} |u(t,z)|^{\alpha+2} \dz.
 	    \end{aligned}
 	\end{equation}    
 It is natural to take the initial data $u_0$ from the following weighted Sobolev space
\begin{equation}
    	\Sigma:= \lbr{ u \in H^1(\R^{d+1}): \nm{ yu }_{L^2(\R^{d+1})}^2:= \int_{\R^{d+1}} y^2 |u|^2 \dz<  +\infty},
\end{equation}
equipped with the norm $\nm{ u }_{	\Sigma }^2= \nm{\nabla_z u }_{L^2(\R^{d+1})}^2+\nm{u }_{L^2(\R^{d+1})}^2+\nm{yu }_{L^2(\R^{d+1})}^2.$  As established in \cite{Ardila-Carles2021},  the solution to \eqref{PHNLS} is locally well-posed in the energy space $\Sigma$. To investigate the global dynamics,  we first introduce some essential definitions.
\vskip0.08in
{\it \noindent $\bullet$ ~~Semivirial functional. }
\vskip0.08in
In the study of long-time dynamics for focusing NLS, Glassey's virial identity plays a crucial role. 
Recall the  standard focusing NLS on $\R^d$
\begin{equation} \label{sm2}
    i\partial_tu+\Delta_x u=-|u|^{\alpha}u, \quad  (t,x)\in\R\times\R^{d}, 
\end{equation} 
which coincides with    \eqref{PHNLS} when the confinement in the  $y$-direction   is removed (so that  $\Delta_z - y^2$ reduces to  $\Delta_x$). 
    Define the virial action functional   $ V (t)$ by 
\begin{equation}
    V (t)=\int_{\R^d} |x|^2|u(t,x)|^2 \dx.
\end{equation}
Glassey \cite{Glassey} established the following  virial identity:
 \begin{equation}  \label{defi of widetilde Q}
    V ''(t)= 8 \widetilde{Q}(u(t)):=8\sbr{\nm{\nabla_x u(t)}_{L^2(\R^{d})}^2 -\frac{\alpha d }{2\sbr{ \alpha+2} }\nm{ u(t)}_{L^{\alpha+2}(\R^{d})}^{\alpha+2}}.
\end{equation}
The quantity $\widetilde{Q}(u(t))$ is called   the virial functional.   Consequently, if the  virial functional is bounded above by some negative value, then the solution will   blow-up  in finite time. Conversely, a positive value of the virial functional may indicate global well-posedness or even scattering, see for instance \cite{KM-Invent,Weinstein=CMP}. 
Inspired by these ideas, when considering the  long-time dynamics of \eqref{PHNLS}, it is natural to  introduce a similar quantity $$\int_{\R^{d+1}} |z|^2|u(t,z)|^2 \dz.$$
 However, due to the  presence of partial harmonic confinement $y^2$, this functional is of limited use for studying the long-time dynamics of \eqref{PHNLS}.      Recall that   the confinement in the $y$-direction is not expected to provide large time dispersion, while the $x$-direction should yield large time dispersion.   It is therefore natural  to consider the dispersive
effects that are  only provided by the $x$-direction, posed on $\R^d$. This leads us to  consider the     {\it semivirial} action functional   $ V_{semi}(t)$ defined by
\begin{equation}
    V_{semi}(t)=\int_{\R^{d+1}} |x|^2|u(t,z)|^2 \dz.
\end{equation}
 A direct calculation implies that 
 \begin{equation}\label{defi of Q}
      V_{semi}''(t)= 8 Q(u(t))=8\sbr{\nm{\nabla_x u(t)}_{L^2(\R^{d+1})}^2 -\frac{\alpha d }{2\sbr{ \alpha+2} }\nm{ u(t)}_{L^{\alpha+2}(\R^{d+1})}^{\alpha+2}}.
 \end{equation}
We call $Q(u(t))$  the {\it semivirial} functional.  As we will see in Theorem \ref{thm:alldimension}, the  condition $Q(\phi)\geq 0$ is sharp for   global existence. 
\vskip0.08in
{\it \noindent $\bullet$ ~~Ground state. }
\vskip0.08in
 To   establish the sharp threshold for the long-time dynamics,  it is essential to introduce the notion of  a ground state. Consider the following elliptic equation 
 \begin{equation}\label{equ elliptic}
		-\Delta_{z} \phi+ y^2 \phi +\omega \phi=  |\phi|^{\alpha}\phi,  \quad z=\sbr{x,y} \in \R^{d} \times \R, 
	\end{equation}
   where  $\omega>0$.   A nontrivial solution $\phi_{\mathrm{GS}} \in \Sigma$ to equation \eqref{equ elliptic} is called a ground state  
  if it  minimizes  the action functional
 $$S_{\omega}(\phi)=E(\phi)+\frac{\omega}{2}M(\phi)$$
 among all nontrivial solutions of the elliptic problem \eqref{equ elliptic}. That is,  
\begin{equation}\label{ml1}
 S_{\omega}(\phi_{\mathrm{GS}})= m_{GS}:=\inf \lbr{ S_{\omega}(\phi) :  \phi \in \Sigma  \setminus \lbr{0} ~~ \text{solves}~~ \eqref{equ elliptic}}.\end{equation} 
 The existence and properties of  ground states $\phi_{\mathrm{GS}} $ can be established via variational methods, exploiting the compactness induced by the harmonic potential in the $y$-direction, see  \cite{Ardila-Carles2021}.
Ground states play a key role in characterizing the threshold for the long-time dynamics.   In this direction, the work of Ardila and Carles \cite{Ardila-Carles2021} established the global dynamics below the ground state energy $m_{GS}$ for dimensions $1 \leq d \leq 4$ under assumption \eqref{assumption}.    However, their framework does not extend to dimensions $d \geq 5$,   a regime that remains open and is considerably more delicate.   In this paper, we overcome this dimensional limitation  by introducing an alternative strategy that is applicable  for all dimensions  $d \geq 3$.

	 \subsection{Main results} \label{Section:mainresult} 
For any $\omega>0$, we define  \begin{equation}
 			\begin{aligned}
 				& \mathcal{K}_{\omega}^+:=\lbr{ \phi \in  	\Sigma \setminus \lbr{0} :   ~ S_{\omega}(\phi)<  m_{GS}, \quad  Q(\phi)\geq 0 },\\
 				& 	\mathcal{K}_{\omega}^-:=\lbr{ \phi \in  	\Sigma \setminus \lbr{0} :  ~S_{\omega}(\phi)<  m_{GS}, \quad Q(\phi)< 0 }.
 			\end{aligned}		 
 		\end{equation}	    
Our main result is the following.

    \begin{theorem}
    \label{thm:alldimension}   
    Assume that   $d\geq 3$   and that \eqref{assumption} holds. Let $u(t)\in C(I,\Sigma )$ be a solution of \eqref{PHNLS}  with initial data $u_0=u(0)$, where $I=(-T_-,T_+)$ denotes the   maximal  lifespan with $T_\pm>0$. 
   Then we have the following dichotomy classification of dynamics.
    \begin{enumerate} \medbreak
      \item If $u_0\in \mathcal{K}_{\omega}^+$, then the solution $u(t)$ is globally well-posed and scatters in the sense that there exists $\varphi_\pm\in \Sigma $ such that 
       \begin{align*}
          \lim_{{t\to \pm\infty}}  \big\|u-e^{it(\Delta_{z}-y^2)}\varphi_\pm\big\|_{\Sigma }=0.
       \end{align*}
     
       \medbreak
       \item If $u_0\in \mathcal{K}_\omega^{-}$, then either $u(t)$ blows-up in finite time $T_\pm<+\infty$ such that 
       \begin{align*}
           \lim_{{t\to \pm T_\pm}}\|\nabla_zu(t)\|_{L^2(\R^{d+1})}=+\infty;
       \end{align*}
       or $u(t)$ blows-up in infinite time, i.e. $T_\pm=+\infty$ and there exists a sequence $\{t_k\}$  with $t_k\to\pm\infty$ such that
       \begin{align*}
           \lim_{t_k\to\pm\infty }\|\nabla_zu(t_k)\|_{L^2(\R^{d+1})}=+\infty.
       \end{align*}
      Moreover, if $|z| u_0\in L^2(\R^{d+1})$, then  $u(t)$  blows-up in finite time. 
   \end{enumerate}
\end{theorem}
\vskip 0.08in
   We note that Ardila and Carles \cite{Ardila-Carles2021} proved  Theorem \ref{thm:alldimension} for dimensions $1\leq d\leq 4$.  Their proof of scattering  is based on the concentration-compactness and rigidity argument pioneered by Kenig and Merle   \cite{KM-Invent}.
This framework, however, does not extend to dimensions $d \geq 5$, because   $|\phi|^\alpha $ fails to be Lipschitz continuous   in higher dimensions.     More precisely, the argument in \cite[Proposition 5.6]{Ardila-Carles2021} requires the condition     $\alpha \geq 1$,  which only holds in  dimensions $1\leq d \leq 4$.   A similar obstruction appears in the study of the standard NLS \eqref{Free-NLS}. To overcome this,   Tao-Visan \cite{Tao-Visan} utilized the fractional calculus to prove a weaker stability result, which is sufficient for the purpose.  However, the anisotropic nature will bring us new difficulties and we do not know whether such weaker stability result holds since the order of derivatives in $x,y$-directions are  different. 
\vskip0.04in
In proving the scattering result  for $d\geq 3$ in Theorem \ref{thm:alldimension},  we introduce an alternative approach     utilizing interaction Morawetz-Dodson-Murphy (IMDM) estimates \cite{DM-MRL} and a coercivity  estimate for the semivirial functional $Q$ (see Lemma \ref{lem coercivity}). A key observation is that,  despite the lack of spatial translation and scaling invariance in the equation, the method is highly compatible with the roadmap developed by Dodson and Murphy \cite{DM-MRL}.   In their work, a scattering criterion is introduced and subsequently verified via an interaction Morawetz estimate. A crucial ingredient in this argument is a suitable coercivity property, where the variational structure of the ground state plays a central role, such as the sharp Gagliardo–Nirenberg inequality and Pohozaev identities. In \cite{DM-MRL}, this coercivity estimate follows from  a refined Gagliardo-Nirenberg inequality (see \cite[Lemma 2.1]{DM-MRL})
\begin{equation}\label{refined-GN}
\|f\|_{L^{\frac{2(d+2)}{d-1}}(\R^d)}^{\frac{2(d+2)}{d-1}}\leq C_{GN}\|f\|_{L^2(\R^d)}\|\nabla f\|^{\frac{2
}{d-2}}_{L^{2}(\R^d)}\|\nabla (e^{ix\cdot\xi}f)\|_{L^2(\R^d)}^{\frac{2d}{d-1}},
\end{equation}
which holds for any $f\in H^1(\R^d)$ and any $\xi\in\R^d$.   
\vskip0.04in
However, this inequality is not applicable in our context. Indeed, due to the presence of the partial harmonic potential $y^2$ and the inhomogeneous nature of the nonlinearity, our problem lacks both translation and scaling invariance, rendering refined estimates such as \eqref{refined-GN} unavailable. This lack of symmetry not only makes the variational problem much more complicated compared to the free case as in \cite{DM-MRL}, but also brings difficulty in proving the interaction Morawetz estimate.    It is here that the semivirial functional $Q$ becomes indispensable, which serves as the key tool that adapts the Dodson-Murphy roadmap to the non-translation-invariant, non-scale-invariant framework of our problem, thereby distinguishing our approach from \cite{Ardila-Carles2021}.

 \vskip0.08in
We now outline the main steps and ideas in the proof of Theorem \ref{thm:alldimension}. The  blow-up result, which is included here to provide a complete characterization of the dynamics of \eqref{PHNLS}, follows from   the Glassey's convex method \cite{Glassey}  and can be obtained by following the arguments in \cite{Ardila-Carles2021}. The scattering result, which constitutes the main contribution of this work, can be proved in four steps.

   \vskip0.04in
 Step 1. To implement the strategy proposed in \cite{DM-MRL}, we first establish a scattering criterion for \eqref{PHNLS} (see Lemma \ref{scattering-cri}).  In our setting, the presence of potential $y^2$ breaks the dispersive estimate. To resolve this, we observe that in the anisotropic norm $L_x^\infty  L_y^2(\R^d\times\R)$, one can capture the dipsersive effect by the $x$-variable. This requires us to  establish such criterion with the anisotropic Strichartz norm.  
 
 \vskip0.04in
 Step 2.  To verify this scattering criterion,  we employ an interaction Morawetz estimate developed by Dodson and Murphy \cite{DM-MRL}, as detailed in Theorem \ref{thm5.1}.    Notably, our choice of multiplier involves only the $x$-direction and therefore differs from that used in \cite{DM-MRL}.  In the proof of Theorem \ref{thm5.1},   a coercivity estimate associated with the virial functional $Q$ plays a crucial role, see Lemma \ref{lem coercivity}. This estimate hinges on the following minimization problem 
\begin{equation} 
 	m_\omega=\inf\lbr{S_{\omega}(\phi):~\phi \in 	 \Sigma  \setminus \lbr{0}, \quad	Q(\phi)=0 }.  
     \end{equation}
     In  Theorem \ref{thm groundstate}, we prove   that this infimum $m_\omega$ is attained by a minimizer  $\phi_\omega$, which is also a solution to \eqref{equ elliptic}.  Solving this variational problem is highly nontrivial,   as  the relation
$Q(\phi) = 0$ is not a natural constraint and introduces certain analytical difficulties. We discuss these challenges in detail in Section \ref{Section Variation}.  
 
 \vskip0.04in
 Step 3.  We introduce a new   set  
 \begin{equation} \label{defi of P}
 			\begin{aligned}
 				& P_{\omega}^+:=\lbr{ \phi \in  	\Sigma  :  ~\phi \neq 0,  \quad	S_{\omega}(\phi)<  m_{\omega}, \quad Q(\phi)\geq 0 }.\\
 			\end{aligned}		 
 		\end{equation}
In  Theorem \ref{thm:main}, by applying the interaction Morawetz estimate, we  prove  that any solution $u(t)$ of \eqref{PHNLS}  is globally well-posed and scatters if  the initial data $ u_0\in P_{\omega}^+$.

\vskip0.04in
 Step 4. It remains to show that $ P_{\omega}^+= \mathcal{K}_{\omega}^+$, or, equivalently, that  $  m_\omega=m_{\mathrm{GS}}$, see Lemma \ref{lemma msequaltomgs}. This equality is not a priori obvious. While Theorem \ref{thm groundstate} provides the inequality  $m_{\mathrm{GS}} \leq m_\omega$,  the reverse inequality $m_\omega \leq m_{\mathrm{GS}}$ does not hold in general because the ground state $\phi_{\mathrm{GS}}$  may fail to satisfy the constraint $Q(\phi_{\mathrm{GS}}) = 0$ due to the presence of  potential $y^2$. To establish this condition,    we introduce a   truncation semivirial functional. More precisely, for any $R>0$, we define
    \[
    V_{R}(t):=\int_{\R^{d+1}}\tilde \varphi_R(x)|e^{i\omega t}\phi_{GS}(z)|^2 dz,
    \] 
 where  $\tilde \varphi_{R}=R^2\tilde \varphi(\frac{x}{R})$. Here, $\tilde \varphi$  is a positive radial smooth function in $C_0^{\infty}(\R^d)$ satisfying  $\tilde \varphi(r)=r^2 $ for  $r=|x|\leq 1$  and $\tilde \varphi(r)''\leq 2 $ for  $r\geq 0$. Obviously, $ V_{R}(t)$ is  independent of $t\in\R$, which implies that $  V_R^{\prime\prime}(t)\equiv 0$. Moreover,   by a direct calculaion, we have   
    \begin{align*}
        V_R^{\prime\prime}(0)=8Q(\phi_{GS})+A_R,
    \end{align*}
    where the remainder term $A_R$ involves the derivatives of  $\tilde \varphi_R(r)$.
    Using the property of the weighted function $\tilde \varphi_R(r)$, we can show that $|A_R|\to0$ as $R\to\infty$, which provides that $$Q(\phi_{GS})=0.$$ This completes the proof.

\vskip0.1in
Finally, we emphasize that in comparison with recent works addressing combined nonlinearities (cf. \cite{BDF-SIAM}), where symmetry-breaking also poses challenges, the variational problem we encounter is substantially more intricate. This increased complexity stems not only from the loss of spatial translation/Galilean invariance induced by the potential, but also from the anisotropic characterization of the partial harmonic oscillator, which forces us to treat the $x$ and $y$ directions with different Sobolev exponents. Consequently, both the resolution of the minimization problem $m_\omega$ and the derivation of the coercivity estimate become comparably more delicate. As a result, our proof is not a straightforward adaptation of \cite{DM-MRL, BDF-SIAM}; rather, it requires overcoming difficulties specific to the partial harmonic confinement and developing new variational tools to bypass these obstacles.

  \begin{remark}{\rm (The NLS on the waveguide manifolds)
Our results can be compared to those for NLS posed on waveguide manifolds   $\R^d\times\Bbb T $, which can be  written as
\begin{align}\label{eq:waveguide}
\begin{cases}
i\partial_tu+\Delta_{z}u=\mu|u|^{\al}u, \quad t\in \R,\\
u(0,z)=  u_0(z),
    \end{cases}
\end{align}
where $z=(x,y)\in \R^d\times\Bbb T$ and $d\geq 1$. Observe that the   compactness  of $\Bbb T$   indicates that any dispersive effects must be purely provided by the $\R^d$-side.   For the defocusing case $\mu=1$, scattering in the range \eqref{assumption} was proved by Tzvetkov-Visciglia \cite{TV2}. In  the focusing case  $\mu=-1$,  Luo \cite{Luo-Arxiv}  considered a variational problem  with   prescribed $L^2$-norm and showed that,  for all $d\geq 1$,  it admits a ground state  that depends only on the dispersive variable  $\R^d$ and that can be trivially extended  on   the   direction $\Bbb T$. By using the classical concentration compactness arguments initiated by \cite{KM-Invent},  a scattering/blow-up dichotomy in  $H^1(\R^d\times\Bbb T)$ is established for $1\leq d\leq4$.  The higher dimensional case $d\geq 5$ was later treated  in \cite{Luo-MA} via the Dodson–Murphy argument. However,  the situation in our case is different. Indeed, due to the  non-compactness  of $\mathbb{R} $, the ground state of \eqref{equ elliptic} cannot be trivially extended  on the $y$-direction while preserving the finite-energy property.  
 
}
    \end{remark}
    \begin{remark}
        {\rm  (The NLS with general harmonic trapping) We consider the   problem 
\begin{equation} \label{generalHNL}
    \begin{cases} 
    i\partial_tu+\Delta_z u-\sum_{j=1}^N y_j^2 u=-|u|^{\alpha}u,&t\in \R,\\
  u(0,z)=  u_0(z),
\end{cases} 
\end{equation} 
where $z=(x,y)\in \R^d\times \R^N$. When there is only one free direction (that is, $d=1$), Hani and Thomann \cite{Hani-Thomann}  used a normal-form approach to construct  a special quasi-periodic solution  for which scattering does not occur.
 It is conjectured that scattering does not occur when $N \gg d$, which remains an interesting open problem.   We would also like to mention that in \cite{Bellazzini=CMP2017,jeanjean-song2025}, the authors investigated the existence, multiplicity, stability and instability of standing waves of \eqref{generalHNL} with the two-dimensional harmonic trap $N=2$ and one free direction $d=1$.}
    \end{remark}
\vskip0.04in

\subsection{Structure of the paper} 
 This paper is organized as follows. In Section 2, we collect the harmonic analysis tools around the partial Hermite operator $H=-\Delta_z+y^2$ and the local Cauchy theory. In Section 3, we show the scattering criterion in the style of Dodson-Murphy \cite{DM-MRL}. In Section 4, we establish the variational characterization of the semi-virial functional and investigate the minimization problem $m_\omega$. In Section 5, We established the interaction Morawetz estimate. In Section 6, we finish the proof of Theorem \ref{thm:alldimension}. More precisely, we utilize the Morawetz estimate in previous section to verify the assumption on scattering criterion, which yield the scattering result for solutions with initial data belonging to $P_\omega^+$. Also, we present the virial analysis to show the equivalence of threshold between ours and the one appeared in Ardila-Carles \cite{Ardila-Carles2021}, that is $m_\omega=m_{GS}$.

	\section{Preliminaries}
	In this section, we collect the basic harmonic analysis tools associated to $-\partial_y^2+y^2$ with $y\in\R$ and the local theory to \eqref{PHNLS}.

    \subsection{Fourier transform and functional spaces}
In this subsection, we recall the basic definition and properties of Fourier transform and some useful functional spaces.

First, we define the Fourier transform of function $f$ as
\begin{align*}
	(\mathcal{F}f)(\xi)=\frac{1}{(2\pi)^{\frac  {d+1}2}}\int_{\R^{d+1}}f(z)e^{-iz\cdot\xi}\,dz,\quad \xi\in\R^{d+1}
\end{align*}
and the inverse Fourier transform of $\mathcal{F}f$  as
\begin{align*}
	 (\mathcal{F}f)^\vee(z)=\frac{1}{(2\pi)^\frac{d+1}{2}}\int_{\R^{d+1}}\widehat{f}(\xi)e^{iz\cdot\xi}d\xi.
\end{align*}
Then, for any $  s\in\R$, we can define the fractional order differential operator $|\nabla|^s$ as $$(|\nabla|^sf)^\wedge(\xi)=|\xi|^s\widehat{f}(\xi).$$ Similarly, the operator $\langle\nabla\rangle^s$ can be defined as $$(\langle\nabla\rangle^sf)^\wedge(\xi)=\langle\xi\rangle^s\widehat{f}(\xi):=(1+|\xi|^2)^\frac{s}{2}\widehat{f}(\xi). $$
In our paper, we also need the Fourier transform and inverse Fourier transform with respect to one variable. For example, we give the definition of Fourier transform on $x$-variable. For $z=(x,y)\in\R^d\times\R$, we define
\begin{align*}
	(\mathcal{F}_{x}f)(\xi,y)=\frac{1}{(2\pi)^\frac d2}\int_{\R^d}f(x,y)e^{-ix\cdot\xi}dx,\quad\xi\in\R^d.
\end{align*}
   Let   $\chi\in C_0^\infty(\R^+,\R)$ be a cut-off function such that $\chi(t)=1$ if $t\leqslant1$ and $\chi(t)=0$ if $t\geqslant2$.  For $N\in2^{\Bbb{Z}}$, we let 
\begin{align*}
	\psi_N(t)=\psi(N^{-1}t),\quad \varphi_N(t)=\psi_{N}(t)-\psi_{N/2}(t).
\end{align*}
Define the full Littlewood-Paley projector as
\begin{align*}
	P_{\leqslant N}f(x)&=\mathcal{F}^{-1}(\psi_N(|\xi|)\widehat{f}(\xi)),\quad x,\xi\in\R^{d+1},\\
	P_Nf(x)&=\mathcal{F}^{-1}(\varphi_N(|\xi|)\widehat{f}(\xi)),\quad x,\xi\in\R^{d+1}, 
\end{align*}
and the partial Littlewood-Paley projector by
\begin{align*}
	P_{\leqslant N}^{x}(z)&=\mathcal{F}_{x}^{-1}(\chi_N(|\xi|)(\mathcal{F}_{x}f)(\xi,y)),\hspace{1ex}x,\xi\in \R^d,\,y\in\R,\\
	P_{ N}^{x}(z)&=\mathcal{F}_{x}^{-1}(\varphi_N(|\xi|)(\mathcal{F}_{x}f)(\xi,y)),\hspace{1ex}x,\xi\in\R^d,\,y\in\R.
\end{align*}
Next, we denote the standard Lebesgue spaces by $L^p(\R^{d+1})$ and its norm is given by
\begin{align*}
\|f\|_{L^p(\R^{d+1})}=\Big(\int_{\R^{d+1}}|f(z)|^p\,dz\Big)^\frac{1}{p}.
\end{align*}  
Let $I$ be the  time interval, we denote the mixed space-time Lebesgue space as $L_t^qL_z^r(I\times\R^{d+1})$ with the norm
\[ 
\|  u \|_{L_{t}^{q}L^{r}_{z}(I\times \R^{d+1})}=\|  \|u(t) \|_{L^{r}_{z}(\R^{d+1})}  \|_{L^{q}_{t}(I)}.
\]
We denote $W^{s,p}$ by the inhomogeneous  Sobolev space,
\begin{align*}
	W^{s,p}(\R^{d+1}):=\left\{f\in L^p(\R^{d+1}):\|f\|_{W^{s,p}(\R^{d+1})}:=\|\langle\nabla\rangle^sf\|_{L^p(\R^{d+1})}<+\infty\right\}.
\end{align*}
	\subsection{Harmonic oscillator and associated functional space}
The Schr\"odinger operator with harmonic potential $H=- \partial_y^2+y^2$ with $y\in\R$   is an important model in mathematics and physics. The one-dimensional harmonic oscillator can form a $L^2$-basis.   We denote $E_n$ by the $n$-th eigenspace related to $H$ and $\lambda_n=2n+1$ the $n$-th eigenvalues. The eigenspace can be generated by the Hermite function $e_n$ which is given by
\begin{align*}
	e_n(y)=\frac{1}{\sqrt{n!}2^{\frac{n}{2}}\pi^\frac{1}{4}}(-1)^ne^{\frac{y^2}{2}}\frac{d^n}{dy^n}(e^{-y^2}),\quad y\in\R,
\end{align*}  
and $e_n$ statisfies the following equation
\begin{align*}
	(-\partial_{y}^2+ y^2)e_n(y)=(2n+1)e_n(y).
\end{align*}
We denote $\Pi_n$ by the spectral projector on the $n$-th eigenspace $E_n$. For any $f\in L^2(\R^{d+1})$, we can decompose it as $$f(z)=\sum\limits_{n\in\N}\Pi_nf(x,y):=\sum_{n\in\N}\langle f,e_n\rangle_{L_{y}^2(\R)} e_n(y),$$
where $$\langle f,e_n\rangle_{L_{y}^2(\R)}=\int_{\R}f(x,y)\overline{e_n(y)}dy.$$ 
For $s\in\R$ and $p\geqslant1$, we denote the inhomogeneous Sobolev space $\mathcal{W}_y^{s,p}(\R)$ by the following
\begin{align*}
	\mathcal{W}_y^{s,p}(\R)=\big\{u\in L_y^p(\R):\|u\|_{\mathcal{W}_y^{s,p}(\R)}:=\|\langle\partial_y\rangle^su\|_{L_y^p(\R)}+\|y^su\|_{L_y^p(\R)}<\infty\big\}.
\end{align*} 
If $p=2$, the above norm is equivalent to the  following, which was proved in \cite{equi}
\begin{align*}
	\|f\|_{\mathcal{W}_y^{s,2}(\R)}^2=\sum_{n\in\N}(2n+1)^s\|\Pi_nf\|_{L_y^2(\R)}^2.
\end{align*}
Throughout  this article, we denote $\mathcal{W}_y^{s,2}(\R)$ by $\mathcal{H}_y^s(\R)$. We denote $L_{x}^p\mathcal{H}_{y}^s(\R^d\times\R)$  by
\begin{align*}L_{x}^{p}\mathcal{H}_{y}^{s}(\R^d\times\R)=\bigg\{f\in L_{x}^{p}L_{y}^{2}&(\R^d\times\R):\|f\|_{L_{x}^{p}\mathcal{H}_{y}^{s}(\R^d\times\R)}:=\bigg(\int_{\R^d}\|f(x,\cdot)\|_{\mathcal{H}_{y}^{s}(\R)}^{p}dx\bigg)^{1/p}\\&=\bigg(\int_{\R^d}\left\|\bigg(\sum_{n\in\mathbb{N}}(2n+1)^{s}|f_{n}(x,y)|^{2}\bigg)^{\frac{1}{2}}\right\|_{L_{y}^{2}(\R)}^{p}dx\bigg)^{1/p}<\infty\bigg\},\end{align*}
where $f_n=\Pi_nf$. 
We can define the Hermite-Sobolev spaces $H_{x}^{s_1}\mathcal{H}_{y}^{s_2}$ and $\dot{H}_{x}^{s_1}\mathcal{H}_{y}^{s_2}$ as 
\begin{align*}
	\|f\|_{H_{x}^{s_1}\mathcal{H}_{y}^{s_2}(\R^d\times\R)}^2=\int_{\R^d}\|\langle\nabla_{x}\rangle^{s_1}f(x,\cdot)\|_{\mathcal{H}_{y}^{s_2}(\R)}^2dx.
\end{align*}
and
\begin{align*}
	\|f\|_{\dot{H}_{x}^{s_1}\mathcal{H}_{y}^{s_2}(\R^d\times\R)}^2=\int_{\R^d}\||\nabla_{x}|^{s_1}f(x,\cdot)\|_{\mathcal{H}_{y}^{s_2}(\R)}^2dx,
\end{align*}
For arbitrary time interval $I\subset\R$ and $f:I\times\R^d\times \R\to\Bbb{C}$, we can define the mixed norm
\begin{align*}
	\|f\|_{L_t^pL_{x}^qW_{y}^{s,r}(I\times\R^d\times\R)}\stackrel{\triangle}{=}\Bigg(\int_{I}\bigg(\int_{\R^d}\big\|\langle\partial_{y}\rangle^sf(t,x,\cdot)\big\|_{L_{y}^r(\R)}^qdx\bigg)^\frac{p}{q}dt\bigg)^\frac{1}{p}
\end{align*}
and
\begin{align*}
\|f\|_{L_t^pL_{x}^q\mathcal{H}_{y}^{s}(I\times\R^d\times\R)}\stackrel{\triangle}{=}\Bigg(\int_{I}\bigg(\int_{\R^d}\big\|f(t,x,\cdot)\big\|_{\mathcal{H}_{y}^s(\R)}^qdx\bigg)^\frac{p}{q}dt\bigg)^\frac{1}{p}
\end{align*}
with $1\leqslant p,q,r\leqslant\infty$. We also use the following norms for sequence $\{f_n\}_{n\in\Bbb{N}}$
\begin{align*}
	\|f_n\|_{L_t^pL_{x}^qL_{y}^r\ell_n^2(I\times\R^d\times\R\times\Bbb{N})}\stackrel{\triangle}{=}\big\|\|f_n\|_{\ell_n^2(\N)}\big\|_{L_t^pL_{x}^qL_{y}^r(I\times\R^d\times\R)}.
\end{align*}
We end this section by giving a Moser type estimate, which is crucial in establishing the nonlinear estimate.
\begin{lemma}[Moser type estimate, \cite{Jao-CPDE}]
	For any $m\geq1$, $\gamma\in(0,1]$ and $p_i,q_i,r\in(1,\infty)$ satisfying $\frac{1}{r}=\frac{1}{p_i}+\frac{1}{q_i}$, $i=1,2$. Then, we have
	\begin{align*}
		\big\|H^\gamma(fg)\big\|_{L^r(\R^d)}\lesssim \|H^\gamma f\|_{L^{p_1}(\R^m)}\|g\|_{L^{q_1}(\R^m)}+\|H^\gamma f\|_{L^{p_2}(\R^m)}\|g\|_{L^{q_2}(\R^m)},
	\end{align*}
where $H^\gamma$ is defined by the functional calculus of the harmonic oscillator $H$.
\end{lemma}

\begin{remark}
    In the rest of  this paper, unless otherwise specified, integrals with respect to $t$,  $x$ and $y$ are over $\R$, $\R^d$ and $\R$,  respectively. For brevity, we suppress the domains of integration in the notation.
\end{remark}
\subsection{Local theory}In this section, we will prove the local control result. We also collect the related result such as local well-posedness. Compared to the nonlinear Schr\"odinger equation without the potential, we use the anisotropic norm as the scattering size. 

Before presenting the local theory, we first recall the Strichartz estimates, which is very crucial in establishing both local and global theory for dispersive equations. We shall use the following notation. 
\begin{definition}[Schr\"odinger admissible pairs] We call $(p,q)\in\R^2$ the admissible pairs if $(p,q)$ satisfies the following 
	\begin{align*}
		\frac{2}{p}=d\Big(\frac{1}{2}-\frac{1}{q}\Big),\quad 2\leqslant p,q<\infty.
	\end{align*}
\end{definition}
Now we can state the Strichartz estimates for linear Schr\"odinger equations with a partially harmonic oscillator. These estimates were proved in Proposition 3.1 in \cite{Antonelli}.
\begin{proposition}[Strichartz estimates, \cite{Antonelli}]\label{Strichartz-Antonelli}
	Let $I$ be a interval,  $(p,q)$ and $(\tilde{p},\tilde{q})$  be any Schr\"odinger admissible pairs,  we have the following estimate
	\begin{align}\label{Strichartz-L^2}
		\big\|e^{it(\Delta_{x,y}-y^2)}f\big\|_{L_t^p(I,  L_{x}^qL_{y}^2)}&\lesssim\|f\|_{L_{x,y}^2 },\\
		\Big\|\int_{0}^{t}e^{i(t-s)(\Delta_{x,y}-y^2)}F(s,x,y)\Big\|_{L_t^p(I  ,L_{x}^qL_{y}^2)}&\lesssim\|F\|_{L_t^{\tilde{p}^\prime}(I  ,L_{x}^{\tilde{q}^\prime}L_{y}^2 )}\label{Strichartz-inhomo},\\
		\label{Stricharz-Sobolev}
		\big\|e^{it(\Delta_{x,y}-y^2)}f\big\|_{L_t^p(I,  L_{x}^q\mathcal{H}_{y}^\sigma) }&\lesssim\|f\|_{L_{x}^2\mathcal{H}_{y}^\sigma },\quad\sigma\geq0.
	\end{align}
	
\end{proposition}
\begin{lemma}[Exotic Strichartz estimates,\cite{TV1,TV2}]Let $d\geq3$ and $\gamma\in\R$. There exists $q,r,q_1,r_1\in(2,\infty)$ such that the following estimate holds
\begin{align}\Big\|\int_{0}^{t}e^{i(t-s)(\Delta_{x,y}-y^2)}F(s,x,y)\Big\|_{L_t^q(I ,L_{x}^r\mathcal H_{y}^\gamma )}&\lesssim\|F\|_{L_t^{\tilde{q_1}^\prime}(I,L_{x}^{\tilde{r_1}^\prime}\mathcal H_{y}^\gamma)}\label{exotic-Strichartz-inhomo},
\end{align}
    where
    \begin{align}
        \frac{1}{q}+\frac{1}{q_1}<1,\,\,\frac{d-2}{d}<\frac{r}{r_1}<\frac{d}{d-2},\label{exotic-cond-1}\\
        \frac{1}{q}+\frac{r}{2}<\frac{d}{2},\,\,\frac{1}{q_1}+\frac{d}{r_1}<\frac{d}{2},\,\,\frac{2}q+\frac{2}{q_1}+\frac{d}{r_1}+\frac{d}{r}=d.\label{exotic-cond-2}
    \end{align}
    If $d=2$, the condition \eqref{exotic-cond-1} can be removed.
\end{lemma}
\begin{remark}
	The Strichartz estimates for Schr\"odinger equation with partially harmonic oscillator is very closed to that in the waveguide manifold setting. The exotic Strichartz estimate can be proved by using eigenbasic expansion in $y$-direction and following the strategies  in \cite{TV1,TV2}. 
\end{remark}

For the pure harmonic oscillator, the effect of the potential implies that the dispersive estimate can only  hold locally-in-time instead of globally-in-time. By Mehler's formula, the associated heat kernel can be written as
 	\begin{align*}
		e^{t(\Delta-|z|^2)}(z,z^\prime)=e^{\alpha(t)(z^2+(z^\prime)^2)}e^{\frac{sinh(t)\Delta}{2}}(z,z^\prime),\quad(z,z^\prime)\in\R^{d+1}\times\R^{d+1},
	\end{align*}
	where $\alpha(t)=\frac{1-cosh(t)}{2sinh(t)}$.  Using the analytic continuation, we can write the solution to the  linear Schr\"odinger equation with partial harmonic oscillator $\Delta_y-y^2$ explicitly,
    \begin{align*}
	e^{it(\Delta_y-y^2)}f(x,y)=\frac{1}{(2\pi i\sin(2t))^\frac{1}{2}}\int_{\R}e^{\frac{i}{\sin(2t)} \big(\frac{y^2+\zeta^2}{2}\cos(2t)- y\cdot \zeta\big)}f(x,\zeta)d\zeta.
\end{align*}
Then one can verify that
\begin{equation}
    	\big\|e^{it( \Delta_{y}-y^2)}f(x,y)\big\|_{L_{y}^\infty }\lesssim|\sin(2t)|^{-\frac{1}{2}}\|f(x,y)\|_{L_{y}^1 },\quad\forall t\notin\frac{\pi}{2}\Bbb{Z}.
\end{equation} 
	However, by adding the strong dispersion of the $x$-direction, we have the global-in-time dispersive estimate:
    \begin{lemma}\label{dispersive}
Let $d\geq2$, we have
\begin{align*}
		\big\|e^{it(\Delta_{x,y}-y^2)}f(x,y)\big\|_{L_{x}^\infty L_{y}^2 }\lesssim|t|^{-\frac{d}{2}}\|f(x,y)\|_{L_{x}^1L_{y}^2 }.
	\end{align*}
    \end{lemma} 
Before stating the local well-posedness result, we  give some notations. We denote $X_1(t)$ and $X_2(t)$   by the following
\begin{align*}
	A_1(t)=y\sin(t)-i\cos(t)\partial_{y},\quad A_2(t)=y\cos(t)+i\sin(t)\partial_{y}.
\end{align*}
For $f\in\mathcal{S}(\R^{d+1})$, we have the following identity:
\begin{align}\label{transform}
	|A_1(t)f(x,y)|^2+|A_2(t)f(x,y)|^2=|yf(x,y)|^2+|\partial_{y}f(x,y)|^2.
\end{align}
By calculating the commutator, we have $[i\partial_t-H,A_j(t)]=0$,  where $H=- \partial_y^2+y^2$.  This implies that when operator $A_j(t)$ acts on $F(u)$, it is similar to take the derivatives. More precise, we have the pointwise bound 
\begin{align}\label{local-est1}
    \big|A_j(t)(|u|^\alpha u)\big|\lesssim |u|^\al|A_j(t)u|.
\end{align}
As a direct consequence, we have the following identity
\begin{align*}
	\|u\|_{\Sigma }^2=\|u\|_{L_{x}^2\mathcal{H}_{y}^1 }^2+\|A_1(t)u\|_{L_{x,y}^2 }^2+\|A_2(t)u\|_{L_{x,y}^2 }^2.
\end{align*}
Now we give the local well-posedness for \eqref{PHNLS}, which was proved in \cite{Ardila-Carles2021}
\begin{lemma}[Local Cauchy theory]\label{lwp}
Let $d\geq3$ and $u_0\in\Sigma$. Assume that \eqref{assumption} holds, then there exists $T=T(\|u_0\|_{\Sigma})>0$ such that \eqref{PHNLS} admits a unique local solution $u\in C([0,T],\Sigma)\cap L_t^qL_x^r\mathcal H_y^{s}$ with $s>\frac12$ and $(q,r)$ is the $\dot H^{\frac{d}{2}-\frac{2}{\al}}$-admissible pair.  Let $T_{max}$ be the maximal lifespan of $u$, then either $T_{max}=\infty$ (global-in-time) or $T_{max}<\infty$ and satisfies 
\begin{align*}
    \|\nabla_z u(t)\|_{L_{x,y}^2(\R^{d+1})}\to\infty,\,\,\mbox{as }t\to T_{max}.
\end{align*}Moreover, if there exists sufficiently small $0<\eta\ll1$ such that $\|u_0\|_{\Sigma}<\eta$, then $u\in C(\R,\Sigma)$ is globally well-posed and scatters in both time directions.
\end{lemma} 
\vskip0.08in
We will then present a local control result, after first establishing a fractional chain rule for the harmonic oscillator.
 
\begin{lemma}[Fractional chain rule]Let $\al>0$ and $s\in[0,1]$, we have the following Moser type estimate
\begin{align}\label{fractional-chain}
    \big\||u|^\al u\big\|_{\mathcal{H}_y^s}\lesssim\|u\|_{\mathcal{H}_y^s}\|u\|_{L_y^\infty}^\al.
\end{align}
\end{lemma}
\begin{proof}
    For $s\in[0,1]$, it suffices to show the two endpoints. When $s=0$, it is a direct consequence of H\"older's inequality. When $s=1$, one should use the pointwise bound \eqref{local-est1} and H\"older. Then by interpolation, we have proved the desired estimate.
\end{proof}
\begin{lemma} [Local control result]
    Let $u$ be a global energy solution to \eqref{PHNLS} with $\|u\|_{L_t^\infty \Sigma}<\infty$ and $s\in(\frac12,1-s_c)$ where $s_c=\frac{d}{2}-\frac{2}{\al}$. Then for any admissible pairs $(p,q)$ with $p<\infty$ and compact interval $I$ with length $|I|$, it holds
    \begin{align*}
        \big\|u\big\|_{L_t^p(I,W_x^{1-s,q}\mathcal{H}_y^s)}\lesssim\langle |I|\rangle^{\frac{1}{p}}.
    \end{align*}
\end{lemma}
\begin{proof}
    First,  we define the norm
    \begin{align*}
        X=\|u\|_{L_t^pW_{x}^{1,q}L_y^2}+\|u\|_{L_t^pL_x^q\mathcal{ H}_y^1}+\|u\|_{L_t^{p_1}W_{x}^{1,q_1}L_y^2}+\|u\|_{L_t^{p_1}L_{x}^{q_1}\mathcal H_y^1}.
    \end{align*}
 Now, we turn to give the control of the norm $X$. Let $r_1>\frac{\al}{d}$ be  such that $H_x^{1-s}\hookrightarrow L_x^{r_1}$, which is implied by $s<1-s_c$. Suppose that $b_1$ satisfies $\frac{2}{b_1}+\frac{d}{r_1}=\frac{2}{\al}.$ Then there exists an admissible pair $(p_1,q_1)$ such that 
\begin{align*}
    \frac{d+2}{2d}=\frac{\al}{r_1}+\frac{1}{q_1},\,\,\frac{1}{2}=\frac{\al}{b_1}+\frac{1}{p_1}.
\end{align*}
Then by Strichartz estimate, H\"older's inequality, Sobolev embedding $H_y^s\hookrightarrow L^\infty$ and the properties of $A_j(t)$, we have
\begin{align*}
    X&\lesssim1+\big\|A_j(t)(|u|^\al u)\big\|_{L_t^2(I,L_x^{\frac{2d}{d+2}}L_y^2 )}\\
    &\lesssim 1+\big\|\|u\|_{L_x^{r_1}\mathcal{H}_y^s}^\alpha\|A_j(t)u\|_{L_x^{q_1}L_y^2}\big\|_{L_t^2(I)}\\
    &\lesssim1+|I|^\frac{\al}{b_1}\|u\|_{L_t^\infty (I,L_x^{r_1}\mathcal{H}_y^s )}^\alpha\|A_j(t)u\|_{L_t^{p_1}(I, L_x^{q_1}L_y^2 )}\\
    &\lesssim1+|I|^\frac{\al}{b_1}\|u\|_{L_t^\infty (I, H_x^{1-s}\mathcal{H}_y^s )}^\alpha\cdot X\lesssim 1+|I|^\frac{\al}{b_1}X.
\end{align*}Notice that we only treat the terms involving $\mathcal{H}_y^1$ norm in $X$, others are more easier to treat by repeating the estimate above. Following the  continuity method, Proposition 2.3 in Carles-Gallo \cite{Carles-Gallo-JMP} and the property of $A_j(t)$,
we have the expected estimate. Denote by $H=-\Delta+|x|^2$ with $x\in\R^d$, we then have $$[L^2,D(H^\frac12)]_{\theta}=D(H^\frac{\theta}{2})=\big\{u\in H^s(\R^d):|x|^su\in L^2(\R^d)\big\}.$$
Combining with the abstract vector-valued interpolation space theorem (cf. Theorem 3.1 in \cite{Amann}) and the equivalence of Sobolev norms (cf. \cite{equi}), we have the following interpolation estimate
\begin{align*}
    \|u\|_{W_x^{1-s,r}\mathcal{H}_y^s}\lesssim\|f\|_{L_x^r\Sigma}^s\|f\|_{W_x^{1,r}L_y^2}^{1-s}
\end{align*}
for $r\in(1,\infty)$. Therefore, we have the desired estimate.

Hence, we complete the proof of this lemma. 
\end{proof}
\section{Scattering criterion}
In this section, we will establish the scattering criterion for \eqref{PHNLS}. This type of scattering criterion was initially developed by Dodson and Murphy in \cite{DM-MRL}. In our setting, the lack of dispersion effect in $y$-direction requires us to establish such criterion with the anisotropic norm.  
\begin{lemma}[Scattering criterion]\label{scattering-cri}
Let $u$ be a global solution to \eqref{PHNLS} such that $$\|u\|_{L_t^\infty \Sigma}\leq B$$ with some $B<\infty$. Then for all $\sigma>0$, there exists $\varepsilon(\sigma,B)>0$ sufficiently small and $T_0(\sigma,\varepsilon,B)$ sufficiently large  so that if for all $a\in\R$ there exists $T\in(a,a+T_0)$ such that  $[T-\varepsilon^{-\sigma},T]\subset(a,a+T_0)$ and 
\begin{align*}
    \|u\|_{L_t^q([T-\varepsilon^{-\sigma},T], L_x^{r}\mathcal{H}_y^s)}\lesssim\varepsilon,
\end{align*}
where $(q,r)=(\frac{2\al(\al+2)}{{2\al+4-d\al}},\al+2)$ is an admissible pair of $H^{s_c}$-level with $s_c=\frac{d}{2}-\frac{2}{\al}$ and $s=1-s_c$,
then $u$ scatters in forward time direction.

\end{lemma}
\begin{proof}
By Lemma \ref{lwp}, it remains to prove that there exists $T>0$ sufficiently large such that the following estimate holds for some $\mu>0$:
\begin{align*}
    \big\|e^{
i(t-T)(\Delta_{x,y}-y^2)} u(T)\big\|_{L_t^qL_x^r\mathcal H_y^s([T,\infty)\times\R^{d+1})}\lesssim\varepsilon^\mu.
\end{align*}
By the standard Duhamel formula, one has
\begin{align*}
    e^{i(t-T)(\Delta_{x,y}-y^2)}u(T)=e^{it(\Delta_{x,y}-y^2)}u_0-i\int_0^Te^{i(t-\tau)(\Delta_{x,y}-y^2)}(|u|^\al u)(\tau)\,d\tau.
\end{align*}
For the linear part, by Strichartz estimate, we find that
\begin{align*}
 \big\|e^{it(\Delta_{x,y}-y^2)}u_0\big\|_{L_t^q L_x^r\mathcal H_y^s}\lesssim\|u_0\|_{\Sigma }.   
\end{align*}
By the continuity, there exists a $T\gg1$ sufficiently large such that 
\begin{align*}
\big\|e^{it(\Delta_{x,y}-y^2)}u_0\big\|_{L_t^q([T,\infty), L_x^r\mathcal H_y^s)}<\varepsilon,
\end{align*}
where $\varepsilon>0$ small enough. Then splitting  the time interval, we rewrite
\begin{align*}
&\hspace{3ex}\int_{0}^Te^{i(t-\tau)(\Delta_{x,y}-y^2)}(|u|^{\al}u)(\tau)\,d\tau\\&=\int_0^{T-\varepsilon^{-\sigma}}e^{i(t-\tau)(\Delta_{x,y}-y^2)}(|u|^{\al}u)(\tau)\,d\tau+\int_{T-\varepsilon^{-\sigma}}^Te^{i(t-\tau)(\Delta_{x,y}-y^2)}(|u|^{\al}u)(\tau)\,d\tau\\
&=:G_1(t)+G_2(t).
\end{align*}
From the Strichartz estimate and \eqref{fractional-chain}, one has
\begin{align*}
    \big\|G_2(t)\big\|_{L_t^q([T,\infty),L_x^r\mathcal{ H}_y^s )}&\lesssim\||u|^{\al}u\|_{L_t^{q_0^\prime}([T-\varepsilon^{-\sigma},T), L_x^{r_0^\prime}\mathcal H_y^s)}\lesssim\|u\|_{L_t^q([T-\varepsilon^{-\sigma},T), L_x^r\mathcal H_y^s)}^{\al+1}\lesssim\varepsilon^{\al+1},
\end{align*}
where $(q_0,r_0)$ is a $H_x^{s_c}(\R^d)$-admissible pairs and satisfies $q_0^\prime(\al+1)=q$ and $r_0^\prime (\al+1)=r$. Then it suffices to estimate the term $G_1(t)$. By interpolation inequality in $x$-direction, one has
\begin{align*}
\|G_1\|_{L_t^qL_x^r\mathcal H_y^s}\leq C\|G_1\|_{L_t^qL_x^{r_1}\mathcal H_y^s}^\frac{r_1}{r}\|G_1\|_{L_t^qL_x^{\infty}\mathcal H_y^s}^{1-\frac{r_1}{r}},
\end{align*}
where $2\leq r_1<\al+2$ and $r_1=\frac{2d\al(\al+2)}{d\al^2+4(d-1)\al-8}$. Hence, $(q,r_1)$ is a $L^2$ admissible pair. 
By Duhamel's formula, we can write
\begin{align*}
    iG_1(t)=e^{i(t-T+\varepsilon^{-\sigma})(\Delta_{x,y}-y^2)}u(T-\varepsilon^{-\sigma})-e^{it(\Delta_{x,y}-y^2)}u_0.
\end{align*}
Using the Strichartz estimate and the uniform bound of $u$ in energy space, we have
\begin{align*}
\|G_1(t)\|_{L_t^qL_x^{r_1}\mathcal{H}_y^s }\lesssim1.
\end{align*}
Next, we proceed on estimating $L_t^qL_x^\infty\mathcal H_y^s$ norm of $G_1(t)$.
By dispersive estimate (Lemma \ref{dispersive}), Strichartz estimate, Moser type estimate, and the Duhamel formula, we have
	\begin{align*}
&\hspace{4ex}\|G_1(t)\|_{ L_x^{\infty}\mathcal{H}_y^s }\\&\lesssim\|G_1\|_{L_x^{\infty}L_y^2}+\big\|H_y^\frac{s}{2}G_1(t)\big\|_{L_x^{\infty}L_y^2}\lesssim\|G_1(t)\|_{L_y^2L_x^{\infty} }+\big\|H_y^\frac{s}{2}G_1(t)\big\|_{L_y^2L_x^{\infty} }\\
			&\lesssim\Big\|\int_0^{T-\varepsilon^{-\sigma}}\big|t-\tau\big|^{-\frac{d}{2}}\big\|e^{i(t-\tau)(\Delta_{y}-y^2)}(|u|^\al u)(\tau)\big\|_{L_x^{1} }\,d\tau\Big\|_{L_y^2 }\\
&\hspace{2ex}+\sum_{j=1}^2\Big\|\int_{0}^{T-\varepsilon^{-\sigma}}\big|t-\tau\big|^{-\frac{d}{2}}\big\|e^{i(t-\tau)(\Delta_{y}-y^2)}H_y^\frac{s}{2}(|u|^\al u)(\tau)\big\|_{L_x^{1} }\,d\tau\Big\|_{L_y^2 }\\
			&\lesssim\int_0^{T-\varepsilon^{-\sigma}}\big|t-\tau\big|^{-\frac{d}{2}}\big\||u|^\alpha u(\tau)\big\|_{L_x^1L_y^2}\,d\tau +\int_{0}^{T-\varepsilon^{-\sigma}}\big|t-\tau\big|^{-\frac{d}{2}}\big\|H_y^\frac{s}{2}(|u|^\al u)(\tau)\big\|_{L_x^{1}L_y^2}\,d\tau\\
            &\lesssim \int_0^{T-\varepsilon^{-\sigma}}\big|t-\tau\big|^{-\frac{d}{2}}\big\|u(\tau)\big\|_{L_x^{\al+1}\mathcal H_y^s}^{\al+1}\,d\tau\\
			&\lesssim \int_0^{T-\varepsilon^{-\sigma}}\big|t-\tau\big|^{-\frac{d}{2}}\|u(\tau)\|_{L_y^2L_x^{
				\al+1} }^{\al+1}\,ds +\sum_{j=1}^2\int_0^{T-\varepsilon^{-\sigma}}\big|t-\tau\big|^{-\frac{d}{2}}\|H_y^\frac{s}{2}(u)\|_{L_y^2L_x^{\al+1} }^{\al+1},d\tau\\
			&\lesssim\int_0^{T-\varepsilon^{-\sigma}}\big|t-\tau\big|^{-\frac{d}{2}}\|u(\tau)\|_{H_x^{s_1}\mathcal{H}_y^{s_2} }^{\al+1}\,d\tau\lesssim\int_0^{T-\varepsilon^{-\sigma}}\big|t-\tau\big|^{-\frac{d}{2}}\|u(\tau)\|_{\Sigma }^{\al+1}\,d\tau\\
			&\lesssim(t-T+\varepsilon^{-\sigma})^{-\frac{d}{2}+1}.
		\end{align*}
        In the last step, we take $s_1=\frac{1}{2}-\eta$ and $s_2=\frac{1}{2}+\eta$ for some $\eta$ close to $0$. 
		Then taking the time integral, we obtain
		\begin{align*}
\big\|G_1(t)\big\|_{L_t^q([T,\infty),L_x^{\infty}\mathcal{ H}_y^1 )}\lesssim\Big(\int_{T}^\infty (t-T+\varepsilon^{-\sigma})^{(-\frac d2+1)q}\,dt\Big)^{\frac1q}\lesssim\varepsilon^{\sigma(\frac{d-2}{2}-\frac{1}{q})}.
		\end{align*}
		It remains to verify the inequality $\frac{d-2}{2}-\frac{1}{q}>0$ for $q=\frac{2\alpha(\alpha+2)}{2\alpha-d\alpha+4}$.
		By the basic calculation, it is equivalent to 
		\[
		(d-2)\alpha^2 + [2(d-2)\alpha - 2\alpha + d\alpha] - 4 = (d-2)\alpha^2 + 3(d-2)\alpha - 4>0.\]
		Then we define the function
$ f(\alpha) = (d-2)\alpha^2 + 3(d-2)\alpha - 4 $. It reduces to proving that when \( \alpha \in \left( \frac{4}{d}, \frac{4}{d-1} \right) \), \( f(\alpha) > 0 \).
		If  $ d \geq 3$, one can easily show that \( f(\alpha) \) is monotone increasing. Then we only need to  verify $f(\frac{4}{d})>0$, which is direct by  simple computation. Therefore, we complete the proof of this lemma. 
\end{proof}

\begin{remark}
   {\rm  It is worth noting that   the $\dot H^{\frac12}$-subcritical nonlinearity in $d$-dimension corresponds precisely to the $\dot H^1$-subcritical nonlinearity in $(d+1)$-dimension. Therefore, we will perform the $H^{\frac12-}$ criticality in $x$-variable while for the $y$ direction we perform the $H^{\frac12+}$ analysis. This is due to the dispersion effect in partial dimensions.  For the waveguide model, this was indicated by Tzvetkov and Visciglia \cite{TV2}.} 
\end{remark}

    \section{Variational analysis}\label{Section Variation}
The goal of this section is to establish some variational results   for  proving Theorem \ref{thm:alldimension}.	Throughout this section, we always assume that     $d\geq 2$  and that condition \eqref{assumption} holds.  
Consider   the   minimization problem 
\begin{equation}\label{minimiztaon problem}
 	m_\omega=\inf\lbr{S_{\omega}(\phi):\phi \in  K},  
     \end{equation}
     where 
     \begin{equation}
         K:=\lbr{\phi \in 	 \Sigma  \setminus \lbr{0}:  ~	Q(\phi)=0 }.
     \end{equation}
     We have the following.

 \begin{theorem}[Existence of  minimizer]\label{thm groundstate}    For any $\omega>0$, the  minimization problem	 \eqref{minimiztaon problem} admits a positive minimizer $\phi_\omega$,  which     solves equation  \eqref{equ elliptic}.  
 	\end{theorem}

In the proof of Theorem \ref{thm groundstate}, the  constraint $Q(\phi)=0$ leads to some unexpected new challenges in the  analysis. First of all,  we point out that a crucial step in    proving the existence of minimizer  is   to  show that the weak limit of a minimizing sequence is  non-vanishing. Consider the    stationary equation of  standard NLS \eqref{sm2}   
\begin{equation} \label{sm3}
   -\Delta_x \phi+\omega \phi=|\phi|^{\alpha}\phi~~ \quad \text{in } \R^d  ,
\end{equation}
and the corresponding minimization problem
\begin{equation}\label{mq1}
    \inf \lbr{  \widetilde{S}_{\omega}(\phi) :  \phi \in H_x^1(\R^{d}) \setminus \lbr{0}, ~~ \widetilde{Q}(\phi)=0 },
\end{equation}
where $\widetilde{Q}(\phi)$ is given by \eqref{defi of widetilde Q}, and  $ \widetilde{S}_{\omega}(\phi)$  is   the corresponding  functional    defined by
\begin{equation}
        \widetilde{S}_{\omega}(\phi)=	\frac{1}{2}\nm{\nabla_x \phi }_{L_x^2(\R^{d})}^2+  \frac{\omega}{2}\nm{\phi }_{L_x^2(\R^{d})}^2 -\frac{1}{\alpha+2}\nm{ \phi}_{L_x^{\alpha+2}(\R^{d })}^{\alpha+2}.
\end{equation}
For the minimization problem \eqref{mq1}, the non-vanishing property of the weak limit of a minimizing sequence can be proved using the Gagliardo–Nirenberg inequality on $\R^{d}$. 
Indeed, let $\sbr{\phi_n}_n $ be a bounded minimizing sequence for \eqref{mq1}. Then, in view of  \eqref{defi of widetilde Q}, it follows that for each  $n\in \N$
\begin{equation}
    \nm{\nabla_x \phi_n}_{L_x^2(\R^{d})}^2 =\frac{\alpha d }{2\sbr{ \alpha+2} }\nm{ \phi_n}_{L_x^{\alpha+2}(\R^{d})}^{\alpha+2}\lesssim \nm{\nabla_x \phi_n}_{L_x^2(\R^{d})}^{\frac{\alpha d}{2}} \nm{\phi_n}_{L_x^2(\R^{d})}^{ \frac{2\alpha+4-\alpha d}{2}}\lesssim \nm{\nabla_x \phi_n}_{L_x^2(\R^{d})}^{\frac{\alpha d}{2}}.
\end{equation}
Since $\alpha d>4$, we obtain  that 
\begin{equation}
    \liminf_{n\to\infty} \nm{\nabla_x \phi_n}_{L_x^2(\R^{d})}^2 \sim   \liminf_{n\to\infty}\nm{ \phi_n}_{L_x^{\alpha+2}(\R^{d})}^{\alpha+2} >0,
\end{equation}
 which  ensures the weak limit is non-vanishing. However, this argument  fails for the minimization problem \eqref{minimiztaon problem} if one uses the      Gagliardo-Nirenberg inequality on  $\R^{d+1}$, because the  constraint $Q(\phi)=0$ does not involve the  term $\nm{\partial_y \phi_n}_{L_{x,y}^2(\R^{d+1})}^2 $.  To overcome this difficulty, we shall use the   anisotropic $(x,y)$-Gagliardo-Nirenberg inequality   for $(x,y)\in \R^d\times\R$
 \begin{equation}
		\nm{\phi}_{L_{x,y}^{\alpha+2}(\R^{d+1})}^{\alpha+2}\leq C(d,\alpha) \nm{ \nabla_x \phi}_{L_{x,y}^2(\R^{d+1})}^{\frac{\alpha d}{2}}\nm{\partial_y \phi}_{L_{x,y}^2(\R^{d+1})}^{\frac{\alpha}{2}}\nm{\phi}_{L_{x,y}^2(\R^{d+1})}^{ \frac{4-\alpha(d-1)}{2}},
		\end{equation}
see Lemma \ref{lemma GNinequality}.  
\vskip0.04in
Another difficulty is to show that the minimizer $\phi_\omega$ solves equation \eqref{equ elliptic} because the condition $Q(\phi)=0$ is not a natural constraint.  We will overcome this difficulty by showing that the functional $S_\omega$ admits a mountain pass structure at level $m_\omega$, see Lemma \ref{lemma moutainpass structure}. Then, adapting a subtle deformation argument from \cite{Bellazzini-Jeanjean-Luo2013} (which deals with  Schr\"odinger–Poisson equations with $L^2$-prescribed norm, so some modifications are necessary here), we prove that $\phi_\omega$ is indeed a solution of \eqref{equ elliptic}, see Lemma \ref{lemma solution}.
\vskip0.04in
Throughout this section, we denote the norm $\nm{\cdot}_{L^p_{x,y}(\R^{d+1})}$ simply by $\nm{\cdot}_p$.
	\subsection{preliminary results}
	\begin{lemma}[Anisotropic $(x,y)$-Gagliardo-Nirenberg inequality]\label{lemma GNinequality}
		 	For $\alpha\in \big(0,\frac{4 }{d-1}\big)$ and   $\phi\in 	H^1(\R^{d+1})$, there exists a constant  $ C(d,\alpha)>0$ such that 
		\begin{equation}
			\nm{\phi}_{\alpha+2}^{\alpha+2}\leq C(d,\alpha) \nm{ \nabla_x \phi}_2^{\frac{\alpha d}{2}}\nm{\partial_y \phi}_2^{\frac{\alpha}{2}}\nm{\phi}_2^{ \frac{4-\alpha(d-1)}{2}}.
		\end{equation}
	\end{lemma}
	\begin{proof}
		It follows from  the anisotropic Gagliardo-Nirenberg inequality (see, for example, \cite{anisotropicGN1,anisotropicGN2}) and Young's inequality that there exists a constant $ C_1(d,\alpha)>0$ such that 
		\begin{equation}
			\begin{aligned}
				\nm{\phi}_{\alpha+2}^{\alpha+2}&\leq C_1(d,\alpha)\nm{\partial_y \phi}_2^{\frac{\alpha}{2}}  \nm{\phi}_2^{\frac{4-\alpha(d-1)}{2}}\prod_{i=1}^d \nm{\partial_{x_i} \phi}_2^{\frac{\alpha}{2}}\\
				&\leq C(d,\alpha)\nm{ \nabla_x \phi}_2^{\frac{\alpha d}{2}}\nm{\partial_y \phi}_2^{\frac{\alpha}{2}}\nm{\phi}_2^{ \frac{4-\alpha(d-1)}{2}},
			\end{aligned}
		\end{equation}
		where $C(d,\alpha)=  C_1(d,\alpha) d^{-\frac{\alpha d}{4}}$.  The proof is complete.
	\end{proof}

	\vskip0.12in
	For any function $\phi \in 	\Sigma$ and a constant $ \lambda>0$, we define   the scaling operator  $ \phi^\lambda$   by
	\begin{equation}\label{scalingoperator}
		\phi^\lambda(x,y)=\lambda^{\frac{d}{2}} \phi \sbr{\lambda x,y}.
	\end{equation}
	In what follows, we establish several important properties of the mappings  $\lambda\mapsto  Q(\phi^\lambda)$, which  will play a crucial role in the proof of our main theorems.
	\begin{lemma} \label{lemma property Q}
		For any $\omega>0$ and $\phi \in 	\Sigma\setminus\lbr{0}$, the following properties hold:
		
		\begin{itemize}	\medbreak
			\item[(i)]   $\frac{\partial }{\partial \lambda}S_\omega(\phi^\lambda)= \lambda^{-1} Q(\phi^\lambda)$ for all $\lambda>0$;
			\medbreak
			\item[(ii)]  There exists a unique $\lambda_\star=\lambda_\star(\phi)$ such that  $	Q(\phi^{\lambda_\star})	=0$. Moreover, 
			\begin{equation}
				Q(\phi^\lambda)\begin{cases}
					>0, \quad \text{ if }~\lambda\in \sbr{0,\lambda_\star},\\
					<0, \quad \text{ if }~\lambda\in \sbr{ \lambda_\star,\infty}.
				\end{cases}
			\end{equation}
			
			\medbreak
			\item[(iii)] We have	the following relations  
			\[
		\lambda_\star > 1 \ \Leftrightarrow\ Q(\phi) > 0; \quad
			\lambda_\star = 1 \ \Leftrightarrow\ Q(\phi) = 0; \quad 
			\lambda_\star < 1 \ \Leftrightarrow\ Q(\phi) < 0.
			\]
			\medbreak
			
			\item[(iv)] $S_\omega(\phi^\lambda) < S_\omega(\phi^{\lambda_\star})  $ for all $ \lambda> 0 $ and $ \lambda \neq  \lambda_\star  $.
		\end{itemize}		 
	\end{lemma}
	
	\begin{proof}
		By a direct calculation, we have  
		\begin{align}\label{defi of E lambda}
			S_\omega(\phi^\lambda)=&\frac{1}{2} \lambda^{2} \nm{\nabla_x \phi }_2^2 -\frac{1}{\alpha+2}   \lambda^{\frac{\alpha d}{2}}\nm{ \phi}_{\alpha+2}^{\alpha+2} +\frac{1}{2} \sbr{ \nm{\partial_y \phi }_2^2   +   \nm{ y\phi }_{2}^2+ \omega \nm{  \phi }_{2}^2 } ,\\\label{defi of Q lambda}
			Q( \phi^\lambda)=& \lambda^{2} \nm{\nabla_x \phi }_2^2 -  \frac{\alpha d }{2\sbr{ \alpha+2} }  \lambda^{\frac{\alpha d}{2}}\nm{ \phi}_{\alpha+2}^{\alpha+2}.
		\end{align}
		Therefore,    the conclusion (i) holds.	 Let $f(\lambda)=	Q( \phi^\lambda)$. Then  
		$$f'(\lambda)= \lambda \sbr{2 \nm{\nabla_x \phi }_2^2 -  \frac{\alpha^2 d^2 }{4\sbr{ \alpha+2} }  \lambda^{\frac{\alpha d-4}{2}}\nm{ \phi}_{\alpha+2}^{\alpha+2}}.$$
		Since $\alpha >4/d$, we can find a unique
		$\lambda_0=\lambda_0(\phi)>0$ such that $ f'(\lambda)$ is positive  on $ \sbr{0,\lambda_0}$ and negative on $\sbr{\lambda_0,\infty}$. Moreover, we deduce from \eqref{defi of Q lambda} that there exists a unique $ \lambda_\star=\lambda_\star(\phi)>\lambda_0$ such that $ f(\lambda_\star)=0$, and  $f (\lambda)$ is positive  on $ \sbr{0,\lambda_\star}$ and negative on $\sbr{\lambda_\star,\infty}$. Hence, the  conclusion (ii) holds. Observe that  $ f(1)=  Q(\phi^1)=Q(\phi)$, then the conclusion (iii)  is a direct consequence of (ii). Finally, to prove (iv), we deduce  from (i) that 
		\begin{equation}
				S_\omega(\phi^{\lambda_\star})= 	S_\omega(\phi^{\lambda}) +\int_{\lambda}^{\lambda_\star}s^{-1} f(s)  {\mathrm{d} s}.
		\end{equation} 
		Hence, conclusion (iv)  follows directly from the facts that $f(s) $ is positive on $ \sbr{0,\lambda_\star}$ and negative on $\sbr{\lambda_\star,\infty}$. The proof is complete.
	\end{proof}
      
			\subsection{The properties of  $m_\omega$}

		\begin{lemma}\label{lemma bounded energy level}
			For any $\omega>0$, we have $m_\omega \in \sbr{0,\infty}$.
		\end{lemma}
		
		\begin{proof}
			First,  we know that $ m_\omega<\infty$ since the set $K\neq \emptyset$. In what follows, we show that $m_\omega>0 $. Indeed, let $ \sbr{\phi_n}_n\subset K $ such that $ S_\omega(\phi_n)=m_\omega+o_n(1)$. Then 
			\begin{equation}\label{se1}
				\begin{aligned}
				m_\omega+o_n(1)&=S_\omega(\phi_n)=S_\omega(\phi_n)-\frac{2}{\alpha d} Q(\phi_n)\\&= \frac{1}{2}    \nm{\partial_y \phi_n }_2^2   +  \frac{1}{2} \nm{ y\phi_n }_{2}^2+ \frac{\omega}{2} \nm{  \phi_n }_{2}^2   +\sbr{\frac{1}{2}-\frac{2}{\alpha d}} \nm{ \nabla_x \phi_n}_2^2 .
				\end{aligned}
			\end{equation}
		 Hence, from $\alpha>4/d$, we know that the sequence  $ \sbr{\phi_n}_n  $ is bounded in $	\Sigma$. Now by Lemma \ref{lemma GNinequality}, we deduce from the fact  $\alpha < 4/(d-1)$  that 
		\begin{equation}\label{lq6}
			\begin{aligned}
				\nm{\nabla_x \phi_n }_2^2 = \frac{\alpha d }{2\sbr{ \alpha+2} }  \nm{ \phi_n}_{\alpha+2}^{\alpha+2} &\lesssim  \nm{ \nabla_x \phi_n }_2^{\frac{\alpha d}{2}}\nm{\partial_y \phi_n }_2^{\frac{\alpha}{2}}\nm{\phi_n }_2^{ \frac{4-\alpha(d-1)}{2}}  \lesssim  \nm{ \nabla_x \phi_{n} }_2^{\frac{\alpha d}{2}}  .
			\end{aligned}
		\end{equation}
		 Combined with  $\alpha>4/d$, this implies the existence of constant $\widetilde{C} =\widetilde{C} (\omega,\alpha,d)>0$  such that 
		 \begin{equation}\label{lq3}
		  \liminf_{n\to\infty}   \nm{ \phi_n}_{\alpha+2}^{\alpha+2} \sim  \liminf_{n\to\infty}\nm{\nabla_x \phi_{n} }_2^2 \geq \widetilde{C} .
		 \end{equation}
		 Therefore, we see from \eqref{se1} that 
		 \begin{equation}
		 	 	m_\omega \geq \liminf_{n\to\infty}\sbr{\frac{1}{2}-\frac{2}{\alpha d}}\nm{\nabla_x \phi_{n} }_2^2 \geq \sbr{\frac{1}{2}-\frac{2}{\alpha d}}\widetilde{C}   >0.
		 \end{equation}
			This completes the proof.
		\end{proof}

		\vskip 0.15in
	 Define \begin{equation}\label{defi of I}
         \begin{aligned}
              		I_\omega(\phi):&=S_\omega(\phi)-\frac{2}{\alpha d}Q(\phi)\\&= \frac{1}{2} \sbr{ \nm{\partial_y \phi }_2^2   +   \nm{ y\phi }_{2}^2+ \omega \nm{  \phi }_{2}^2 }  +\sbr{\frac{1}{2}-\frac{2}{\alpha d}} \nm{ \nabla_x \phi_n}_2^2.
         \end{aligned}		
		 	\end{equation}
		
		 \begin{lemma} \label{lemma Le Coz}
		 For any $\omega>0$, we define \begin{equation}\tilde{m}_\omega:=\inf \lbr{ I_\omega(\phi): ~\phi\in 	\Sigma\setminus\lbr{0}, ~Q(\phi) \leq 0}.
 \end{equation}
		 	Then $m_\omega =\tilde{m}_\omega$.

		 \end{lemma}
		 \begin{proof}
		 	Obviously, we have 
		 	\begin{equation}\label{pq1}
		 	\tilde{m}_\omega \leq \inf \lbr{  I_\omega(\phi):\phi \in K}=\inf \lbr{ S_\omega(\phi): \phi \in K}=	m_\omega<\infty.
		 	\end{equation}
		 	On the other hand, let  $\sbr{\phi_n}_n\subset \Sigma\setminus\lbr{0}$ be a minimizing sequence  for $\tilde{m}_\omega$,  i.e.,  
		 	\begin{equation} \label{p2}
		 		\lim_{n\to\infty} I_\omega(\phi_n)=\tilde{m}_\omega<\infty\quad \text{and} \quad Q(\phi_n) \leq  0, \quad \forall n\in \N.
		 	\end{equation}	
		 	By Lemma \ref{lemma property Q}, there exists $\lambda_n\in (0,1]$ such that $Q(\phi_n^{\lambda_n} )=0$.  Then 
		 	\begin{equation}
		 		\begin{aligned}
		 			m_\omega\leq S_\omega( \phi_n^{\lambda_n} )=I_\omega(\phi_n^{\lambda_n} )\leq I_\omega(\phi_n)=\tilde{m}_\omega+o_n(1).
		 		\end{aligned}
		 	\end{equation}
		 	Taking the limit as $n\to \infty$, we obtain $ 	m_\omega\leq \tilde{m}_\omega$, which together with \eqref{pq1}, gives that $m_\omega =\tilde{m}_\omega$. The proof is complete.
		 \end{proof}
 
	\vskip0.15in

In what follows,	we show that  the functional $S_\omega$ has a mountain pass structure   at level $m_\omega$. 
	\begin{definition}
		The functional $S_\omega $ is said to have a mountain pass structure   at level $\sigma_\omega $ if there exist some constants $ a>0 $ and  $ 0<\delta<m_\omega$ such that 
		\begin{equation}
			\sigma_\omega :=\inf_{\gamma\in\Gamma_\omega} \max_{s\in \mbr{0,1}} S_\omega(\gamma(s))>\max\lbr{ \sup_{\gamma\in\Gamma_\omega}  S_\omega(\gamma(0)), \sup_{\gamma\in\Gamma_\omega}  S_\omega(\gamma(1)) },
		\end{equation}
		where 
		\begin{equation}
			\Gamma_\omega:=\lbr{ \gamma\in C([0,1],\Sigma):  \gamma(0) \in A_{a,\delta} , ~ S_\omega(\gamma(1))<0 },
		\end{equation}
		and 
		\begin{equation}
			A_{a,\delta}:=\lbr{ \phi\in \Sigma: \nm{\nabla_x \phi}_2^2<a,~~\nm{\partial_y \phi}_2^2+\nm{ y\phi }_{2}^2+\omega\nm{\phi }_{2}^2 <2(m_\omega- \delta) }.
		\end{equation}
	\end{definition}
	\vskip0.1in
	\begin{lemma} \label{lemma Adelta Q}
		For any $\omega>0$, there exists a constant $a_0>0$   (independent of $\delta$)  such that for $a\in\sbr{0,a_0}$,  the inequality $Q(\phi)>0$ holds for  all $\phi\in 	A_{a,\delta}$. 
	\end{lemma}
	\begin{proof}
		 Since   $ \nm{\partial_y \phi}_2^2+\nm{ y\phi }_{2}^2+\omega\nm{\phi }_{2}^2<2m_\omega$ and  $\alpha>4/d$,  we deduce from Lemma \ref{lemma GNinequality} that
		\begin{equation}
			Q(\phi)=\nm{\nabla_x \phi }_2^2 -  \frac{\alpha d }{2\sbr{ \alpha+2} }   \nm{ \phi}_{\alpha+2}^{\alpha+2} \geq \nm{\nabla_x \phi }_2^2  -C\nm{ \nabla_x \phi}_2^{\frac{\alpha d}{2}}>0,
		\end{equation}
		provided 
		$ \nm{\nabla_x \phi}_2^2<a$ with  $a\in\sbr{0,a_0}$ for  sufficiently small $a_0$.
	\end{proof}
	\vskip0.1in
	\begin{lemma}[Mountain pass structure]\label{lemma moutainpass structure}
		For any $\omega>0$,	there exists   $ 0<\delta<m_\omega$ such that for any $  0<a<\min\lbr{a_0,\delta} $, we have  $$	m_\omega=	\sigma_\omega,$$ and the functional $S_\omega$ has a mountain pass structure    at level $m_\omega$. 
	\end{lemma}
	\begin{proof}
		First, we   prove that $m_\omega=	\sigma_\omega$.  Let   
		$ \sbr{\phi_n}_n\subset K$ be a minimizing sequence such that $ S_\omega(\phi_n)=m_\omega+o_n(1)$. Similar to \eqref{se1} and \eqref{lq3}, we deduce   that  $\sbr{\phi_n}_n$   is bounded in $\Sigma$, and   there exists  a constant $\zeta=\zeta(c,d,\alpha)$  such that 
		\begin{equation}\label{l4}
			\liminf_{n\to\infty}\sbr{\frac{1}{2}-\frac{2}{\alpha d}} 	\nm{\nabla_x \phi_n }_2^2 \geq \zeta.
		\end{equation} 
		We may further assume that  $\zeta$  is small such that $\zeta<4m_\omega$.  For $n$ sufficiently large, we have 
		\begin{equation}
		m_\omega+\frac{1}{4}\zeta \geq  S_\omega(\phi_n), \quad \text{ and } \quad \sbr{\frac{1}{2}-\frac{2}{\alpha d}} 	\nm{\nabla_x \phi_n }_2^2 \geq \frac{1}{2}\zeta.
		\end{equation}
		Combining this with \eqref{se1},  we obtain
		\begin{equation}
				m_\omega+\frac{1}{4}\zeta \geq  S_\omega(\phi_n)\geq \frac{1}{2} \sbr{ \nm{\partial_y \phi_n }_2^2   +   \nm{ y\phi_n }_{2}^2+ \omega \nm{  \phi_n }_{2}^2 }+\frac{1}{2}\zeta.
		\end{equation}
		Set 
		\begin{equation}\label{defi of delta}
			\delta= {\zeta}/{4}\in  (0,	m_\omega) ~ \text{ and } ~0<a<\min\lbr{a_0,\delta} .
		\end{equation}
		Then for any fixed $n$ sufficiently large, we have  $$ \nm{\partial_y \phi_n}_{2}^2 + \nm{ y\phi_n}_{2}^2 + \omega \nm{  \phi_n }_{2}^2 <2(m_\omega-\delta).$$ Since $$\nm{\partial_y \phi_n^\lambda}_{2}^2 + \nm{ y\phi_n^\lambda}_{2}^2+ \omega \nm{  \phi^\lambda_n}_{2}^2= \nm{\partial_y \phi_n}_{2}^2 + \nm{ y\phi_n}_{2}^2 +\omega \nm{  \phi_n}_{2}^2<2( m_\omega-\delta),$$ and $ \nm{\nabla_x \phi_n^\lambda }_2^2 =\lambda^{2} \nm{\nabla_x \phi_n }_2^2 \to0$ as $\lambda\to 0$, we may choose $\lambda_0>0$ small such that $\phi_n^{\lambda_0} \in A_{a,\delta} $ for $a\in \sbr{0,a_0}$.  Moreover, since $\alpha>4/d$,  it follows from \eqref{defi of E lambda} that $ S_\omega(\phi_n^{\lambda }) \to -\infty$ as $\lambda\to +\infty$. Thus, we may find $\lambda_1>0$  large enough such that $S_\omega(\phi_n^{\lambda_1}) <0$. Define 
		\begin{equation}
			\gamma_n(s):=\phi_n^{s\lambda_1+(1-s)\lambda_0 }, \quad  s\in \mbr{0,1}.
		\end{equation}
		Then $\gamma_n(s)\in \Gamma_\omega$. Therefore,  since $Q(\phi_n)=0$, we deduce from Lemma \ref{lemma property Q} that  
		\begin{equation}
			\sigma_\omega \leq \max_{s\in \mbr{0,1}} S_\omega(\gamma_n(s)) \leq S_\omega(\phi_n^1) =S_\omega(\phi_n)=m_\omega+o_n(1).
		\end{equation}
		Letting $n\to \infty$, we obtain that $		\sigma_\omega  \leq m_\omega$. 
		\medbreak
		Next, we  prove the inverse inequality. For any $\gamma\in \Gamma_\omega$, we have $\gamma(0)\in A_{a,\delta}$. Hence, by Lemma \ref{lemma Adelta Q}, we know that $Q(\gamma(0))>0$. We now prove that $Q(\gamma(1))<0$. By contradiction, we assume that $Q(\gamma(1))\geq 0$. Then we have 	 
		\begin{equation}
		 S_\omega(\gamma(1))\geq \frac{1}{2}  \nm{\nabla_x  \gamma(1) }_2^2 -\frac{1}{\alpha+2}   \nm{ \gamma(1)}_{\alpha+2}^{\alpha+2} \geq \sbr{\frac{1}{2} -\frac{2}{\alpha d}} \nm{\nabla_x  \gamma(1) }_2^2 \geq 0 .
		\end{equation}  
	However,  by the definition of $\Gamma_\omega$, we know that  $	 S_\omega(\gamma(1))<0$,	which is a contradiction. Thus,  there holds $Q(\gamma(1))<0$. By the continuity of $\gamma$, there exists $s_0\in \sbr{0,1}$ such that $Q(\gamma(s_0))=0$. Therefore, we have 
		\begin{equation}
			m_\omega\leq  S_\omega( \gamma(s_0))\leq \max_{s\in \mbr{0,1}}  S_\omega( \gamma(s)),
		\end{equation}
		from which we obtain $	m_\omega\leq 	\sigma_\omega$. Hence, we finish the proof of  $	m_\omega=	\sigma_\omega$.
		\medbreak
		Finally, 
			by using  the definition of $  A_{a,\delta}$ and the choice of $a$, we know that   for any $\phi\in A_{a,\delta}$, there holds
		\begin{equation}
			 S_\omega(\phi)\leq \frac{1}{2}  \nm{\nabla_x \phi }_2^2 +\frac{1}{2} \sbr{ \nm{\partial_y \phi }_2^2   +   \nm{ y\phi }_{2}^2+ \omega \nm{  \phi }_{2}^2 } \leq m_\omega-\delta +\frac{1}{2} a \leq m_\omega-\frac{1}{2}\delta .
		\end{equation}
		For any $\gamma\in \Gamma_\omega$, by using Lemma \ref{lemma bounded energy level} and the facts $\gamma(0)\in A_{a,\delta}$, $ S_\omega(\gamma(1))<0$,  we obtain
		\begin{equation}
			\sigma_\omega=	m_\omega>\max\lbr{ \sup_{\gamma\in\Gamma_\omega}  S_\omega(\gamma(0)), \sup_{\gamma\in\Gamma_\omega}  S_\omega(\gamma(1)) }.
		\end{equation}
		Hence, the functional $ S_\omega$ has a mountain pass structure   at level $m_\omega$.  This completes the proof.
	\end{proof}

		\subsection{The proof of Theorem \ref{thm groundstate}}
   This subsection is devoted to the proof of Theorem \ref{thm groundstate}. We begin with two lemmas.

 	\begin{lemma}\label{lemma minimizerexistence}
 	For any $\omega>0$,  the minimization problem $m_\omega$ admits a non-negative minimizer   $\phi_\omega\in \Sigma$.
 \end{lemma}
	\begin{proof}
			Let $\sbr{\varphi_n}_n\subset K$ be a minimizing sequence  for  $m_\omega$, i.e., 
		\begin{equation}
			\lim_{n\to\infty} S_\omega(\varphi_n)=m_\omega\quad \text{and} \quad Q(\varphi_n)= 0, \quad \forall n\in \N.
		\end{equation}  By diamagnetic inequality we
		know that the variational problem  $m_\omega$ is stable under the mapping  $\varphi \mapsto |\varphi|$, thus we may assume
		that  $\varphi_n\geq 0$ for all $n\in \N$.
		Following the same arguments in  Lemma \ref{lemma bounded energy level},  the sequence   $\sbr{\varphi_n}_n$ is bounded in $\Sigma$.  Moreover,  it satisfies   the non-vanishing property:
		\begin{equation}
			\liminf_{n\to\infty} \nm{\varphi_n}_{\alpha+2}^{\alpha+2}>0.
		\end{equation} 
	Then by	   \cite[Lemma 3.4]{Bellazzini=CMP2017},  there exists a sequence $ \sbr{x_n}_n\in \R^{d}$ and  $ \phi_\omega\in \Sigma\setminus\lbr{0}$ such that 
		\begin{equation}
			\phi_{n } (x,y):=\varphi_n(x-x_n,y)\rightharpoonup \phi_\omega \quad \text{ weakly in } ~\Sigma.
		\end{equation} 
		Obviously, we have 
	\begin{equation}
		\lim_{n\to\infty} S_\omega(\phi_n)=m_\omega\quad \text{and} \quad Q(\phi_n)= 0, \quad \forall n\in \N.
	\end{equation}
 By the    weak lower semi-continuity of norms and the definition of $I_\omega$ in \eqref{defi of I}, we deduce from   Lemma \ref{lemma Le Coz} that 
	\begin{equation}\label{lqv1}
		I_\omega(\phi_\omega )\leq   \liminf_{n \to \infty} I_\omega(\phi_n)=\liminf_{n \to \infty} S_\omega(\phi_n)=m_\omega=\tilde m_\omega.
	\end{equation}
	 	Now we prove  that $Q(\phi_\omega)=0$. We divide the proof into two cases:
	 	\vskip0.12in
	 	{\it Case 1.}  If the case $Q(\phi_\omega) < 0$ happens,  we deduce from   Lemma \ref{lemma property Q} that  there exists $\lambda_\star \in \sbr{0,1}$ such that $Q( \phi_\omega^{\lambda_\star }) =0$. Then  by \eqref{lqv1}, we have 
	 	\begin{equation}
	 		\begin{aligned}
	 			\tilde m_\omega\leq  I_\omega(\phi_\omega^{\lambda_\star} )< I_\omega(\phi_\omega)\leq  	\tilde m_\omega,
	 		\end{aligned}
	 	\end{equation}
	 	which is a contradiction.  Therefore, Case 1 is impossible.
	 	
	 	\vskip0.12in
	 	{\it Case 2.} If the case $Q(\phi_\omega) >0$ happens, we consider  $w_n:=\phi_n-\phi_\omega$. By using the Br\'ezis-Lieb lemma \cite{Brezis-Lieb lemma}, a direct calculation shows that 
	 	\begin{equation}\label{p6}
	 		\begin{aligned}
	 			&	Q(w_n)= Q(\phi_n  )-Q(\phi_\omega)+o_n(1) ,\\
	 			&  I_\omega(w_n)=I_\omega(\phi_n)-I_\omega(\phi_\omega)+o_n(1) ,
	 		\end{aligned}
	 	\end{equation}
	 	Therefore, for $n$ large enough, we have $Q(w_n)<0 $. By   Lemma \ref{lemma property Q}, there exists $s_n \in (0,1)$ such that $Q(w_n^{s_n} )=0$.  Hence,  we deduce from   \eqref{lqv1}, \eqref{p6}   that 
	 	\begin{equation}
	 		\begin{aligned}
	 			\tilde		m_\omega  \leq  I_\omega(w_n^{s_n}  ) < I_\omega(w_n)= I_\omega(\phi_n)-I_\omega(\phi_\omega)+o_n(1)\leq\tilde  m_\omega-I_\omega(\phi_\omega)+o_n(1). 
	 		\end{aligned} 
	 	\end{equation}
	 	This implies that $I_\omega(\phi_\omega)=0$, which is absurd. Therefore, Case 2 is impossible.
	 	\medbreak
	 	Since Cases 1 and 2 are both impossible, we have   $Q(\phi_\omega)=0$. Therefore, we obtain from \eqref{lqv1} that 
	 	\begin{equation}
	 		m_\omega\leq S_\omega(\phi_\omega )=I_\omega(\phi_\omega )\leq   m_\omega.
	 	\end{equation}
	 Thus $\phi_\omega$  is a non-negative minimizer of $m_\omega$. This completes the proof.
		\end{proof}

	\begin{lemma}\label{lemma solution}
	  Let   $\phi_\omega$ be the  minimizer as in  Lemma \ref{lemma minimizerexistence}. Then $\phi_\omega$ is a positive solution to  \eqref{equ elliptic}.
	\end{lemma}
	\begin{proof}
	  Indeed, the statement that  $\phi_\omega$ solves \eqref{equ elliptic}  is equivalent to the condition $ S_\omega'(\phi_\omega)=0$.  We proceed by contradiction. Suppose instead that $\nm{   S_\omega'(\phi_\omega)}_\star\neq 0$, where 
	  	\begin{equation}
	  \nm{  S_\omega'(\phi_\omega)}_\star =\sup_{\psi\in  \Sigma, ~\nm{ \psi}_{\Sigma}\leq 1 }	|\hbr{ S_\omega'(\phi_\omega),\psi }   |.
	  \end{equation} Then there exist constants $\rho>0$ and $\kappa>0$ such that 
		\begin{equation}\label{l1}
			\nm{ S_\omega'(\phi)}_\star\geq \kappa, \quad \text{ for any }\phi \in B_{3\rho}(\phi_\omega),
		\end{equation}
		where $B_{3\rho}(\phi_\omega):=\lbr{ \phi \in \Sigma: \nm{ \phi-\phi_\omega}_{\Sigma}\leq 3\rho } $. Let $a$ and $\delta$ be the constants given by Lemma \ref{lemma moutainpass structure} and we may further assume that $0\leq \delta\leq \frac{d}{2}-\frac{2}{\alpha}$. Then we define 
		\begin{equation}\label{se6}
			\begin{aligned}
				&\delta_1:= \frac{1}{4}\sbr{ m(c)-\max\lbr{ \sup_{\gamma\in\Gamma_\omega} S_\omega(\gamma(0)), \sup_{\gamma\in\Gamma_\omega} S_\omega(\gamma(1)) }},\\
				&\delta_2:=\min\lbr{\delta_1, \frac{\rho\kappa}{4}, \frac{m_\omega }{4}} .
			\end{aligned}
		\end{equation}
		Consider  the function $  L: \Sigma\to [0,\rho] $ given by 
		\begin{equation}\label{se4}
			L(\phi)=\frac{\rho ~\dist{\phi,\Sigma\setminus S_1}}{\dist{\phi,\Sigma\setminus S_1}+\dist{\phi,S_2}}\geq 0,
		\end{equation}
		where
		\begin{align*}
			S_1 &:= \Sigma \cap S_\omega^{-1}\big([m_\omega - 2\delta_2, m_\omega + 2\delta_2]\big), \\
			S_2 &:= B_{3\rho}(\phi_\omega) \cap S_\omega^{-1}\big([m_\omega - \delta_2, m_\omega  + \delta_2]\big).
		\end{align*} 
		Let 
		\begin{equation}
			\widetilde{H}_\omega:=\lbr{ \phi\in \Sigma: ~ S_\omega'(\phi)\neq 0}.
		\end{equation}
	Then there exists a locally Lipschitz pseudo gradient vector field  $X: \widetilde{H}_\omega \to   \Sigma $ such that  for any $\phi \in \widetilde{H}_\omega$, we have  $X(\phi) \in \Sigma $ and 
		\begin{equation}\label{se3}
			\nm{ X(\phi)}_{\Sigma}\leq 2\nm{  S_\omega'(\phi)}_\star \quad \text{ and } \quad 	\hbr{ S_\omega'(\phi),X(\phi)  } \geq \nm{ S_\omega'(\phi)}_\star^2.
		\end{equation}	
		We now define the vector field $Y: \Sigma \to \Sigma$ by
		\begin{equation}
			Y(\phi)=\begin{cases}
				-	L(\phi) \nm{X(\phi)}_{\Sigma}^{-1} X(\phi) , \quad &\text{ if } \phi \in \widetilde{H}_\omega,\\
				0, \quad &\text{ if } \phi \in \Sigma\setminus\widetilde{H}_\omega,
			\end{cases}
		\end{equation}
		which is   locally Lipschitz continuous.  Then the  problem 
		\begin{equation}\label{se2}
			\frac{	d}{d \tau} \eta(\tau,\phi)=Y(\eta(\tau,\phi)), \quad \eta(0,\phi)=\phi, \quad \forall \phi \in \Sigma
		\end{equation}
		admits a unique solution $\eta:\R\times\Sigma  \to \Sigma $. We now verify the following properties of $\eta$:
		\begin{itemize}
			\medbreak
			\item[(i)]  $\eta(\tau,\phi)= \phi$ for all $\tau\in \R$ if $\phi\in S_3:= \Sigma  \setminus S_\omega^{-1}\sbr{\mbr{m_\omega-2\delta_2, m_\omega+2\delta_2}}$; 
			\medbreak\item[(ii)] $S_\omega(\eta(\tau,\phi))\leq S_\omega(  \phi)$ for all $\phi\in \Sigma $ and $\tau\in \mbr{0,1}$;
			\medbreak
			\item[(iii)] $\eta (1, S_\omega^{m_\omega+\delta_2}\cap B_\rho(\phi_\omega)) \subset S_\omega^{m_\omega-\delta_2}$. 	\medbreak
		\end{itemize}
		Here, $S_\omega^M$ denotes 
		the set $ S_\omega^M:=\lbr{ \phi\in \Sigma:S_\omega(\phi)\leq M}$. 
		 
		\vskip 0.1in
		{\it Proof of (i)}.
		For any $ \phi\in  S_3$,  we have $L(\phi)=0$ , and thus $Y(\phi) = 0$.   By \eqref{se2}, it follows that $$ \frac{	d}{d \tau} \eta(\tau,\phi)|_{\tau=0}=Y(\eta(0,\phi))=Y(\phi)=0,$$ which implies  $\eta(\tau,\phi)=\eta(0,\phi)=\phi$.
		\vskip 0.1in
		{\it Proof of (ii)}.  For any $ \phi\in \Sigma$, from   \eqref{se4} and \eqref{se3}, a direct calculation shows that
		\begin{equation}\label{se5}
			\begin{aligned}
				S_\omega(\eta(\tau,\phi))&= S_\omega(\eta(0,\phi))+\int_{0}^{\tau} \frac{	d}{d s} S_\omega\sbr{\eta(s,\phi)} ~\mathrm{d} s\\
				&=S_\omega(\phi)+\int_{0}^{\tau}\hbr{S_\omega'\sbr{\eta(s,\phi)}, Y(\eta(s,\phi))}~\mathrm{d} s\\
				&=S_\omega(\phi)-\int_{s\in\mbr{0,\tau},\eta(s,\phi)\in 	\widetilde{H}_\omega}\hbr{S_\omega'\sbr{\eta(s,\phi)}, 	L(\eta(s,\phi)) \nm{X(\eta(s,\phi))}_{\Sigma}^{-1} X(\eta(s,\phi))  }~\mathrm{d} s\\
				&\leq S_\omega(\phi)-\frac{1}{2}\int_{s\in\mbr{0,\tau},\eta(s,\phi)\in 	\widetilde{H}_\omega}L(\eta(s,\phi))\nm{ S_\omega'(  \eta(s,\phi))}_\star~\mathrm{d} s\\
				&\leq S_\omega(\phi). 
			\end{aligned}
		\end{equation}

		\medbreak
		{\it Proof of (iii)}.  Let  $\phi \in S_\omega^{m_\omega+\delta_2}\cap B_\rho(\phi_\omega)$. Then  
		\begin{equation}\label{l2}
			S_\omega(\phi)\leq m_\omega+\delta_2,\quad \text{ and } \quad \nm{\phi- \phi_\omega}_{\Sigma} \leq \rho.
		\end{equation}
		If $ S_\omega(\phi) <m_\omega-\delta_2$, then by (ii),  we have 
		\begin{equation}
			S_\omega(\eta(1,\phi))\leq S_\omega(\phi)<  m_\omega-\delta_2,
		\end{equation}
		so  $	S_\omega(\eta(1,\phi)) \in S_\omega^{m_\omega-\delta_2}$. Now we assume that $ S_\omega(\phi) \in  [m_\omega-\delta_2,m_\omega+\delta_2 ] $. Then   we deduce from (ii)  that
		\begin{equation}
			S_\omega(\eta(\tau,\phi))\leq S_\omega(\phi)<  m_\omega+\delta_2, \quad \forall \tau\in \mbr{0,1}.
		\end{equation}
		Moreover, for all $\tau\in [0,1]$,  we have the estimate:
		\begin{equation}
			\begin{aligned}
				\nm{ \eta(\tau,\phi)-\phi}_{\Sigma} &=\nm{\int_{0}^\tau Y(\eta(s,\phi)) ~\mathrm{d}s}_{\Sigma} \leq \int_{0}^\tau L(\eta(s,\phi))~\mathrm{d}s\leq \rho,  \quad \forall \tau\in \mbr{0,1}.
			\end{aligned}
		\end{equation}
		where we used the fact $0\leq L(\cdot)\leq \rho$. Combining this with \eqref{l2}, we obtain    $\eta(\tau,\phi)\in B_{2\rho}(\phi_\omega)$ for all $\tau\in [0,1]$. It follows from \eqref{l1} that  
		\begin{equation}\label{l3}
			\nm{  S_\omega'(\eta(\tau,\phi))}_\star\geq \kappa, \quad \forall \tau\in \mbr{0,1}.
		\end{equation}
		Define $$\mathcal{E}:=\lbr{\tau\in [0,1]: m_\omega-\delta_2\leq 	S_\omega(\eta(\tau,\phi)) \leq m_\omega+\delta_2}\neq \emptyset.$$ For any $\tau\in\mathcal{E} $, we know that $\eta(\tau,\phi) \in S_2$, and consequently  $L(\eta(\tau,\phi))=\rho$. Combining this with   \eqref{se6}, \eqref{se5} and \eqref{l3},  we conclude that 
		\begin{equation} 
			\begin{aligned}
				S_\omega(\eta(1,\phi))&\leq S_\omega(\phi)-\frac{1}{2}\int_{s\in\mathcal{E},~\eta(s,\phi)\in \widetilde{H}_\omega}L(\eta(s,\phi))\nm{  S_\omega'(  \eta(s,\phi))}_\star~\mathrm{d} s\\
				& \leq m_\omega+\delta_2-\frac{\rho \kappa}{2} \leq  m_\omega-\delta_2.
			\end{aligned}
		\end{equation} 
		Hence, $	S_\omega(\eta(1,\phi)) \in S_\omega^{m(c)-\delta_2}$. The proof of (iii) is complete.
		\medbreak
		Let 
		\begin{equation}
			\gamma(s):=\phi_\omega^{s\lambda_1+(1-s)\lambda_0 }, \quad  s\in \mbr{0,1}.
		\end{equation}
		We aim to find  there exist $ \lambda_0, \lambda_1\geq 0$ such that $\gamma\in \Gamma_\omega$. Similar to the proof of  Lemma \ref{lemma moutainpass structure}, we only need to show that  
		\begin{equation}\label{p1}
			\nm{\partial_y \phi_\omega}_2^2+\nm{ y\phi_\omega }_{2}^2+\omega\nm{  \phi_\omega }_{2}^2<2(m_\omega-\delta).
		\end{equation}
		Indeed, recall that $S_\omega(\phi_\omega)=m_\omega$ and $Q(\phi_\omega)=0$, we have 
		\begin{equation}
			\begin{aligned}
				m_\omega&=S_\omega(\phi_\omega)= S_\omega(\phi_\omega)-\frac{2}{\alpha d}Q(\phi_\omega)\\&= \frac{1}{2} \sbr{ \nm{ \partial_y \phi_\omega}_{2}^2 +   \nm{ y\phi_\omega}_{2}^2 +\omega\nm{  \phi_\omega }_{2}^2 }+\sbr{\frac{1}{2}-\frac{2}{\alpha d}} \nm{ \nabla_x \phi_\omega}_2^2 .
			\end{aligned}
		\end{equation}
	Similar to \eqref{l4}, we deduce from \eqref{defi of delta} that  $\sbr{\frac{1}{2}-\frac{2}{\alpha d}} 	\nm{\nabla_x \phi_\omega}_2^2 \geq \zeta>\delta$, which directly implies \eqref{p1}.
	Since $ Q(\phi_\omega)=0$, it follows from  Lemma \ref{lemma property Q}  that  $S_\omega(\gamma(s))\leq S_\omega(\phi_\omega)=m_\omega$ for all $s\in \mbr{0,1}$, from which we get $\gamma(s)\in S_\omega^{m_\omega}  $.  We now consider three cases:
		\begin{itemize}
			\vskip0.1in
			\item[Case 1:]   If $\gamma(s) \in \Sigma\setminus B_\rho(\phi_\omega)$, then by (ii), we have 
			\begin{equation}
				S_\omega(\eta(1,\gamma(s))) \leq 	S_\omega( \gamma(s)) <  S_\omega(\phi_\omega)=   m_\omega.
			\end{equation} \vskip0.1in
			\item[Case 2:] If $\gamma(s) \in     S_\omega^{m_\omega-\delta_2}$, then by (ii), we have 
			\begin{equation}
				S_\omega(\eta(1,\gamma(s))) \leq 	S_\omega( \gamma(s)) \leq 	m_\omega-\delta_2<m_\omega. 
			\end{equation}\vskip0.1in
			\item[Case 3:] If $\gamma(s) \in   B_\rho(\phi_\omega) \cap   S_\omega^{-1} \sbr{[m_\omega- \delta_2, m_\omega ]} $, then by (iii), we have 
			\begin{equation}
				S_\omega(\eta(1,\gamma(s))) \leq 	m_\omega-\delta_2<m_\omega. 
			\end{equation} 
		\end{itemize} 
		In summary, we know that
		\begin{equation}\label{q1}
			\max_{s\in \mbr{0,1}}	S_\omega(\eta(1,\gamma(s)))<m_\omega.
		\end{equation}
		Finally, we claim that
		\begin{equation}\label{q2}
			\eta(1,\gamma(s))\in \Gamma_\omega.
		\end{equation}
		Once this claim is proved, we are done. Indeed, 	combining \eqref{q1} and \eqref{q2} with Lemma \ref{lemma moutainpass structure}, we arrive at
		\begin{equation}
			\gamma_\omega \leq 	\max_{s\in \mbr{0,1}}	S_\omega(\eta(1,\gamma(s)))<m_\omega=\gamma_\omega,
		\end{equation}
		which is a contradiction. 	
		To prove \eqref{q2}, we first 
		observe that   the choice of $\delta_2$ ensures that $m_\omega-2\delta_2\geq  {m_\omega}/{2}>0$. Then  $ S_\omega(\gamma(1))<0<m_\omega-2\delta_2$, it follows that $  \gamma(1) \in S_3$. Using conclusion (i), we  deduce that  $S_\omega(\eta(1,\gamma(1)))=S_\omega(\gamma(1))<0$. On the other hand, the selection of $\delta_1$ and $\delta_2$ yields
		\begin{equation}
		 S_\omega	(\gamma(0))\leq m_\omega-4\delta_1\leq  m_\omega-4\delta_2< m(c)-2\delta_2,
		\end{equation}
		which implies that $ \gamma(0)\in S_3$.  Using (i) again, we obtain that $\eta(1,\gamma(0))=\gamma(0)\in  A_{a,\delta}$. This establishes \eqref{q2}.		
		We have thus shown that $\phi_\omega$ is a solution to \eqref{equ elliptic}.  The positivity of $\phi_\omega$ follows directly from the  maximum principle. This completes the  proof.
	\end{proof}

\begin{proof}[Proof of Theorem \ref{thm groundstate}]
    The proof of Theorem \ref{thm groundstate}   is just a combinition of Lemmas \ref{lemma minimizerexistence} and \ref{lemma solution}.
\end{proof}

\section{Interaction Morawetz estimate}
In this section, we devote to prove the following estimate, which plays key role in the proof of Theorem \ref{thm:main}. Here we choose the multiplier that is only related to  $x$-direction, which differs from that in \cite{DM-MRL}. Additionally, as mentioned in the introduction, the coercivity estimates required for the proof (Lemma \ref{lem coercivity}) rely heavily on the variational results we established earlier, rather than being a simple consequence of the Gagliardo-Nirenberg inequality.  Let $\eta>0$ be a small constant and let $\chi:\R^d\to \R^+\cup\lbr{0}$ be a  radial decreasing function
  defined by
\begin{equation}\label{defchi}
\chi(x)=\left\{\begin{array}{ll}
1, & |x| \leq 1-\eta, \\
0, & |x|>1,
\end{array} \right. 
\end{equation}
and $\chi_R(x)=\chi(x / R)$ for any $R>0$. 
\begin{theorem}\label{thm5.1}
  Let $u$ be a  solution to \eqref{PHNLS} with  initial data  $u_0\in P_{\omega}^{+}$.
        Then for $\varepsilon>0$ sufficiently small, there exist $T_0=T_0(\varepsilon,u), J_0=J_0(\varepsilon,u), R_0=R_0\left(\varepsilon, u\right)$ sufficiently large, and $\eta=\eta(\varepsilon)>0$ sufficiently small such that for any $a \in \mathbb{R}$,
        \begin{equation}\label{thm5.1e}
            \begin{aligned}
    &\frac{1}{{J_0}T_0}\int_a^{a+T_0}\int_{R_0}^{R_0e^{J_0}}\frac{1}{R^d}\iiint \big|\chi_R(x_2-s)u(t,z_2)\big|^2\times\\
    &\hspace{20ex}\times\big|\nabla_x(\chi_R(x_1-s)u^\xi(t,z_1))\big|^2\,dz_1\,dz_2\,ds\frac{dR}{R}\,dt\lesssim\varepsilon,
\end{aligned}
        \end{equation} 
where $u^{\xi}(t, z)=e^{i x \cdot \xi} u(t, z)$ with $\xi = \xi(t, s, R) \in \mathbb{R}^d$ and 
\begin{equation*}
\xi=\xi(t,s,R)=\left\{
\begin{array}{rr}
             -\frac{\mathrm{Im} \int \chi_R^2(x-s)\bar{u}(t,z)\nabla_x u(t,z)\,dz}{\int \chi_R^2(x-s)|u(t,z)|^2\,dz}, &
             \text{if }\int \chi_R^2(x-s)|u(t,z)|^2\,dz\neq 0,\\
             0, &\text{if }\int \chi_R^2(x-s)|u(t,z)|^2\,dz=0.
             \end{array}
\right.
\end{equation*}
Henceforth we abbreviate $\int_{\mathbb{R}^{d+1}}f\,dz$ by $\int f\,dz$ and $\int_{\mathbb{R}^{d}}g \,ds$ by $\int g\,ds$. 
\end{theorem}

Before proving the theorem, we begin by discussing some basic properties of solutions in $P_{\omega}^+$.
 
  \begin{lemma}\label{global1}
Let $u$ be a solution of  \eqref{PHNLS}  and assume that there exists some $t$ in the lifespan of $u$ such that $u(t) \in P_{\omega}^+$. Then $u(t) \in P_{\omega}^+$ for all $t$ in the maximal lifespan of $u$.
\end{lemma}
\begin{proof}
    We argue by contradiction. Suppose that there exists some time $s$ in the lifespan of $u$ such that $Q(u(s))< 0$. Then by continuity of $u$ and conservation of energy, we know that there exists some $t^{\prime}$ lying between $t$ and $s$ such that $Q(u(t^{\prime}))=0$ and $S_{\omega}(u(t^{\prime}))<m_\omega$, which obviously contradicts the definition of $m_\omega$.
\end{proof}

\begin{lemma}\label{global2}
Let $u$ be a solution of  \eqref{PHNLS} with initial data     $u_0\in P_{\omega}^+$. Then $u$ is global and satisfies
\begin{equation}\label{xxx}
    \sup_{t\in\R}\|u(t)\|_{\Sigma}\lesssim  m_{\omega}.
\end{equation}
\end{lemma}
\begin{proof}
   By Lemma \ref{global1}, $u(t)\in P_{\omega}^+$ for all $t$ in the maximal lifespan of $u$. Therefore, by the definition of $P_{\omega}^+$, we have 
    \begin{align*}
    \|u(t)\|_{\Sigma}&\lesssim \frac{1}{2}    \|\partial_y u(t)\|_2^2   +  \frac{1}{2} \|yu(t)\|_{2}^2+ \frac{\omega}{2} \|u(t)\|_{2}^2   +(\frac{1}{2}-\frac{2}{\alpha d}) \|\nabla_x u(t)\|_2^2\\
        &=S_{\omega}(u(t))-\frac{2}{\alpha d}Q(u(t))<m_\omega,
    \end{align*}
which together with Lemma \ref{lwp} yields  global well-posedness and \eqref{xxx}.
\end{proof}
We now prove the following coercivity result, which is crucial for the proof of Theorem \ref{thm5.1}.
\begin{lemma}[Coercivity estimate]\label{lem coercivity}
Let    $u$ be a solution of \eqref{PHNLS} with initial data  $u_0\in P_{\omega}^+$. Then there exist two time-independent constants $0<\delta\ll 1$ and $R_0\gg 1$ such that for all $R\geq R_0$, $s\in \mathbb{R}^d$, $t\in\mathbb{R}$, it holds
\begin{equation}
Q(\chi_R(\cdot-s)u^\xi(t))\geq \delta\|\nabla_x(\chi_R(\cdot-s))u^\xi(t)\|^2_{2}.\label{coer}
\end{equation}

\end{lemma}
\begin{proof}
Assume that $S_{\omega}(u)=m_\omega-\nu$  for some $\nu>0$. Since
\begin{align*}
\int|\nabla_{x}(\chi u)|^2\,dz&=\int\chi^2|\nabla_x u|^2\,dz-\int \chi \Delta_x \chi |u|^2\,dz,\\
\int|\partial_y(\chi u)|^2\,dz&=\int\chi^2|\partial_y u|^2\,dz,
\end{align*}
we obtain
\begin{align*}
\int|\nabla_{x}(\chi u^\xi)|^2\,dz&=|\xi|^2\int\chi^2|u|^2\,dz+\int\chi^2|\nabla_xu|^2\,dz\nonumber\\
&-\int \chi \Delta_x \chi |u|^2\,dz+2\xi\cdot\int\mathrm{Im}(\chi^2 \bar{u}\nabla_x u)\,dz,\\
\int|\partial_y(\chi u^\xi)|^2\,dz&=\int\chi^2|\partial_y u|^2\,dz.
\end{align*}
Let 
\[
\tilde{I}_\omega(u(t)):=I_\omega(u(t))-\frac{1}{2}\|u(t)\|_{2}^2-\frac{\omega}{2}\|yu(t)\|_{2}^2.
\]
Then if $\int \chi_R^2(x-s)|u(t,x,y)|^2\,dxdy\neq 0$, we have
\begin{align}
\tilde{I}_\omega(\chi_R(x-s) u^\xi(t))&=\frac{1}{2}\int \chi_R^2(x-s)|\partial_y u(t)|^2\,dz+\Big(\frac{1}{2}-\frac{2}{\alpha d}\Big)\int \chi_R^2(x-s)|\nabla_{x} u(t)|^2\,dz\nonumber\\
&-\frac{\Big(\int\mathrm{Im}(\chi_R^2(x-s)\bar{u}(t,z)\nabla_x u(t,z))\,dz\Big)^2}{\int \chi_R^2(x-s)|u(t,z)|^2\,dz}
\nonumber\\
&-\int \chi_R(x-s)\Delta_x(\chi_R(x-s))|u(t,z)|^2\,dz\nonumber\\
&\leq \tilde{I}_\omega(u(t))+O(R^{-2}).\label{5.3}
\end{align}
In the case $\int \chi_R^2(x-s)|u(t,z)|^2\,dz= 0$, we can  deduce \eqref{5.3} in a similar way. Taking $R_0\gg 1$ such that $O(R^{-2})\leq \frac{\nu}{2}$ for all $R\geq R_0$, and combining with $S_{\omega}(u)=m_\omega-\nu$ and $Q(u(t))>0$, we derive 
\begin{align*}
    I_\omega(\chi_R(x-s) u^\xi(t))&\leq \tilde{I}_\omega(\chi_R(x-s) u^\xi(t))+\frac{1}{2}\|u(t)\|_{2}^2+\frac{\omega}{2}\|yu(t)\|_{2}^2\\
    &\leq I_\omega(u(t))+\frac{\nu}{2}+\frac{1}{2}\|u(t)\|_{2}^2+\frac{\omega}{2}\|yu(t)\|_{2}^2\\
    &\leq S_{\omega}(u(t))-\frac{2}{\alpha d}Q(u(t))+\frac{\nu}{2}\\
    &\leq m_\omega-\frac{\nu}{2}.
\end{align*}
In view of Lemma \ref{lemma Le Coz}, we have  $Q(\chi_R(x-z) u^\xi(t))>0$. Define $\Theta:=\chi_R(x-z)u^\xi(t),$ we consider two distinct cases that may occur.\medbreak
\begin{itemize}
    \item[ Case 1:] If
\[
4\|\nabla_x \Theta\|_2^2-\frac{\alpha d(\alpha d+4)}{4(\alpha+2)}\|\Theta\|_{\alpha+2}^{\alpha+2}\geq 0 
\]
happens,  we have
\begin{align}
Q(\Theta)\geq\left(1-\frac{8}{\alpha d+4}\right) \|\nabla_x \Theta\|_2^2.
\end{align}
As $\alpha>\frac{4}{d}$, we have $1-\frac{8}{\alpha d+4}>0$. Thus, \eqref{coer} follows naturally.
\medbreak
\item[ Case 2:] If
\begin{align}\label{hyp2}
4\|\nabla_x \Theta\|_2^2-\frac{\alpha d(\alpha d+4)}{4(\alpha+2)}\|\Theta\|_{\alpha+2}^{\alpha+2}< 0
\end{align}
happens, we define $f(\lambda):=E(\Theta^{\lambda})$, where $\Theta^{\lambda}$ is defined by \eqref{scalingoperator}. Then by a direct calculation, we see that
\begin{align}
(Q(\Theta^{\lambda}))'&=(\lambda f'(\lambda))'=-2f'(\lambda)+\lambda\left(4\|\nabla_x \Theta\|_2^2-{\lambda}^{\frac{\alpha d}{2}-1}\frac{\alpha d(\alpha d+4)}{4(\alpha+2)}\|\Theta\|_{\alpha+2}^{\alpha+2}\right)\\&=-2f'(\lambda)+\lambda h(\lambda).
\end{align}
Using \eqref{hyp2} we know that $h(\lambda)<0$ for all $\lambda\in[1,\infty)$, thus
\begin{align}
(Q(\Theta^{\lambda}))'\leq -2f'(\lambda)\quad\text{for all $t\in[1,\infty)$}.\label{5.7}
\end{align}
Since $Q(\Theta)>0$, by Lemma \ref{lemma property Q} there exists   $\lambda_0\in(1,\infty)$ such that $Q(\Theta^{\lambda_0})=0$. Moreover, by Lemma \ref{lemma Le Coz}, we see that $I_\omega(\Theta^{\lambda_0})\geq m_\omega$. Therefore, integrating \eqref{5.7} yields
\begin{align}
Q(\Theta)&\geq 2(S_{\omega}(\Theta^{\lambda_0})-S_{\omega}(\Theta))=2\left(I_\omega(\Theta^{\lambda_0})-I_\omega(\Theta)-\frac{2}{\alpha d}Q(\Theta)\right),
\end{align}
which combining with $I_\omega(\Theta)\leq m_\omega-\frac{\nu}{2}$ implies
\begin{align}
Q(\Theta )&\gtrsim \nu.
\end{align}
On the other hand, as $Q(\Theta)>0$,  Lemma \ref{global2} guarantees that $\|\Theta\|_{\Sigma} \lesssim m_\omega$,
accordingly, there holds that 
\[
Q(\Theta)\gtrsim \|\nabla_{x}\Theta \|_{2}^2,
\]
as desired.
\end{itemize}
\end{proof}	
With the above preparation in hand, we now in the position to prove Theorem \ref{thm5.1}.
\begin{proof}[Proof of Theorem \ref{thm5.1}]
We will prove Theorem \ref{thm5.1} via the virial-Morawetz identity. To this end, we first introduce some functions.  Let $ R \gg 1 $ be sufficiently large and let $ \phi $ and $ \varphi $  be radial functions satisfying
\begin{equation*}
\phi_R(x) = \frac{1}{\omega_{d}R^{d}} \int_{\mathbb{R}^d} \chi^2 \left(\frac{x - s}{R}\right) \chi^2 \left(\frac{s}{R}\right) {\rm d}s,
\end{equation*}
and
\begin{equation*}
\varphi_R(x) = \frac{1}{\omega_{d}R^{d}} \int_{\mathbb{R}^d} \chi^{\alpha+2} \left(\frac{x - s}{R}\right) \chi^2 \left(\frac{s}{R}\right) {\rm d}s,
\end{equation*}
where $ \omega_{d} $ is the volume of unit ball in $ \mathbb{R}^d $.
Finally, we define
\begin{equation*}
\psi_R(x) = \frac{1}{\vert x \vert} \int_{0}^{\vert x \vert} \phi_R(r) {\rm d}r.
\end{equation*}
We next define the interaction Morawetz quantity
\begin{equation}
M(t) = 2 \int_{\mathbb{R}^{d+1}} \int_{\mathbb{R}^{d+1}} \operatorname{Im} \overline{u(t,x_2,y_2)} \nabla u(t,x_2,y_2) \cdot \psi_R(x_1-x_2)(x_1-x_2) |u(t,x_1,y_1)|^2 {\rm d}z_1 {\rm d}z_2.\label{M}
\end{equation}
By  a direct  calculation, we have
\begin{align}
&\frac{d}{dt}M(t) = -4 \iint \partial_{x_j} \left( \operatorname{Im} \left( \bar{u}(t, z_2) \partial_{x_j} u(t, z_2) \right) \right) \nonumber \\&\quad \times \psi_R(x_1 - x_2) (x_1 - x_2)_k \left( \operatorname{Im} \left( \bar{u}(t, z_1) \partial_{x_k} u(t, z_1) \right) \right) dz_1 dz_2 \label{estx1} \\
&- 4 \iint |u(t, z_2)|^2 \psi_R(x_1 - x_2) (x_1 - x_2)_j \nonumber \\
&\quad \times \partial_{x_k} \left( \operatorname{Re} \left( \partial_{x_k} \bar{u}(t, z_1) \partial_{x_j} u(t, z_1) \right) \right) dz_1 dz_2 \label{estx2} \\
&+ \iint |u(t, z_2)|^2 \psi_R(x_1 - x_2) (x_1 - x_2) \cdot \nabla_x \Delta_x \left( |u(t, z_1)|^2 \right) dz_1 dz_2 \label{estx3} \\
&+ \frac{2\alpha}{\alpha+2} \iint|u(t, z_2)|^2 \psi_R(x_1 - x_2) (x_1 - x_2) \cdot \nabla_x \left( |u(t, z_1)|^{\alpha+2} \right) dz_1 dz_2 \label{estx4} \\
&- 4 \iint \operatorname{Im} \left( \bar{u}(t, z_2) \partial_y^2 u(t, z_2) \right)  \psi_R(x_1 - x_2) (x_1 - x_2) \cdot \operatorname{Im} \left( \bar{u}(t, z_1) \nabla_x u(t, z_1) \right) dz_1 dz_2 \label{esty1} \\
&- 2 \iint |u(t, z_2)|^2 \psi_R(x_1 - x_2) (x_1 - x_2) \cdot \operatorname{Re} \left( \partial_y^2 \bar{u}(t, z_1) \nabla_x u(t, z_1) \right) dz_1 dz_2 \label{esty2} \\
&+ 2 \iint |u(t, z_2)|^2 \psi_R(x_1 - x_2) (x_1 - x_2) \cdot \operatorname{Re} \left( \bar{u}(t, z_1) \partial_y^2 \nabla_x u(t, z_1) \right) dz_1 dz_2 \label{esty3} \\
&- 2 \iint |u(t, z_2)|^2 \psi_R(x_1 - x_2) (x_1 - x_2) \cdot \nabla_x \left[ |y_1|^2 |u(t, z_1)|^2 \right] dz_1 dz_2. \label{esty4}
\end{align}
We observe that since $\bar{u} \partial_y^2 u=\partial_y\left(\bar{u} \partial_y u\right)-\left|\partial_y u\right|^2$ and $\psi_{R}$ is $y$-independent,  it then follows from integrating by parts,
  \begin{equation}\label{aest1}
      \eqref{esty1}=0 
  \end{equation}
 and
 \begin{equation*}\label{aest2}
     \eqref{esty2}+\eqref{esty3}=0.
 \end{equation*}
Moreover, as 
\begin{equation}\label{compu1}
    \sum_{k=1}^d\partial_{x_k}\left[\psi_R(x) x_{k}\right]=d \phi_R(x)+(d-1)(\psi_R-\phi_R)(x)
\end{equation}
and $\psi_R-\phi_R\geq0$, by integrating by parts, we also have
\begin{align}
\eqref{esty4}&=2d\iint\left|u\left(t, z_2\right)\right|^2 \phi_R\left(x_1-x_2\right) \left|y_1\right|^2\left|u\left(t, z_1\right)\right|^2 d z_1 d z_2\notag\\
&+2(d-1)\iint\left|u\left(t, z_2\right)\right|^2 (\psi_R-\phi_R)\left(x_1-x_2\right) \left|y_1\right|^2\left|u\left(t, z_1\right)\right|^2 d z_1 d z_2\notag\\
&\geq 0.\label{aest3}
\end{align}

Next we treat the terms that related to the $x$-direction. Let us denote $P_{j k}(x_1-x_2)=\delta_{j k}-\frac{(x_1-x_2)_j(x_1-x_2)_k}{|x_1-x_2|^2}$. Then, by performing calculations that are virtually identical to those in \cite{DM-MRL}, we can obtain that
\begin{equation}\begin{aligned}
&\eqref{estx1}=  -\frac{4}{\omega_d R^d} 
\iiint\chi^2\left(\frac{x_1-s}{R}\right) \chi^2\left(\frac{x_2-s}{R}\right) \operatorname{Im}[\bar{u} \nabla_x u](t,z_1) \cdot \operatorname{Im}[\bar{u} \nabla_x u](t,z_2) d z_1 d z_2 d s \label{subest1} \end{aligned}
\end{equation}
\begin{equation}
    \begin{aligned}
&\hspace{3ex}\hspace{3ex} -4
\iint\operatorname{Im}\left(\bar{u} \partial_{x_j}u\right)(t,z_2) \operatorname{Im}\left(\bar{u} \partial_{x_k}u\right)(t,z_1) P_{j k}(x_1-x_2)[(\psi_R-\phi_R)(x_1-x_2)] d z_1 d z_2,\label{subest2}
    \end{aligned}
\end{equation}
\begin{equation}
    \begin{aligned}
  & \eqref{estx2}=   \frac{4}{\omega_d R^d} 
\iiint\chi^2\left(\frac{x_1-s}{R}\right) \chi^2\left(\frac{x_2-s}{R}\right)|u(t,z_2)|^2|\nabla_x u(t,z_1)|^2 d x d y d s \label{subest3}      
    \end{aligned}
\end{equation}
  \begin{equation}
      \begin{aligned}
& \hspace{3ex}+4 
\iint|u(t,z_2)|^2 \operatorname{Re}\left(\partial_{x_j}\bar{u} \partial_{x_k}u\right)(t,z_1) P_{j k}(x_1-x_2)[(\psi_R-\phi_R)(x_1-x_2)] d z_1 d z_2,\label{subest4}          
      \end{aligned}
  \end{equation}

\begin{equation}
\begin{aligned}
\eqref{estx3}=\iint|u(t,z_2)|^2 \nabla_x|u(t,z_1)|^2 \cdot \nabla_x[(d-1) \psi_R(x_1-x_2)+\phi_R(x_1-x_2)] d z_1 d z_2,\label{subest8}
\end{aligned}
\end{equation}
and
\begin{align}
\eqref{estx4}= & -\frac{2\alpha d}{(\alpha+2) \omega_d R^d} 
\iiint\chi^2\left(\frac{x_2-s}{R}\right) \chi^{\alpha+2}\left(\frac{x_1-s}{R}\right)|u(t,z_2)|^2|u(t,z_1)|^{\alpha+2} d z_1 d z_2 d s \label{subest5}\\
& -\frac{2\alpha(d-1)}{\alpha+2} 
\iint|u(t,x_2)|^2|u(t,x_1)|^{\alpha+2}[\psi_R-\phi_R](x_1-x_2) d z_1 d z_2 \label{subest6}\\
& -\frac{2\alpha d}{\alpha+2} 
\iint |u(t,x_2)|^2|u(t,x_1)|^{\alpha+2}\left[\phi_R-\varphi_R\right](x_1-x_2) d z_1 d z_2 .\label{subest7}.
\end{align}

Let $\slashed{\nabla}_{x_j}$ denote the angular derivative (related to the $x$-direction) centered at $x_j$, $j=1,2$. Then following the strategy in \cite{DM-MRL},  
\begin{align}
   \eqref{subest2}+\eqref{subest4}= & 4 
   \iint |u(t,z_2)|^2|\left.\slashed{\nabla}_{x_2} u(t,z_1)\right|^2[(\psi_R-\phi_R)(x_1-y_1)] d z_1 d z_2 \\
& -4 
\iint\operatorname{Im}\left[\bar{u} \slashed{\nabla}_{x_1} u\right](t,z_2) \cdot \operatorname{Im}\left[\bar{u} \slashed{\nabla}_{x_2} u\right](t,z_1)[(\psi_R-\phi_R)(x_1-x_2)] d z_1 d z_2
\end{align}
and hence by Cauchy-Schwarz and the fact that $\psi_R-\phi_R\geq0$, we deduce
\begin{align}\label{aest4}
    \eqref{subest2}+\eqref{subest4}\geq 0.
\end{align}

We turn to \eqref{subest1}+\eqref{subest3}. As observed in \cite{DM-MRL}, one may check that for any $\xi\in\R^d$, if we define $u^{\xi}:=e^{ix\cdot\xi}u$, then
\begin{align*}
    &
    \iint \chi^2\left(\frac{x_1-s}{R}\right) \chi^2\left(\frac{x_2-s}{R}\right)\left\{|u(t,z_2)|^2|\nabla_x u(t,z_1)|^2-\operatorname{Im}[\bar{u} \nabla_x u](t,z_1) \cdot \operatorname{Im}[\bar{u} \nabla_x u](t,z_2)\right\} d z_1 d z_2 \\
    &=
    \iint \chi^2\left(\frac{x_1-s}{R}\right) \chi^2\left(\frac{x_2-s}{R}\right)\Big\{|u^{\xi}(t,z_2)|^2|\nabla_x u^{\xi}(t,z_1)|^2\\
    &\hspace{32ex}-\operatorname{Im}[\overline{u^{\xi}} \nabla_x u](t,z_1) \cdot \operatorname{Im}[\overline{u^{\xi}} \nabla_x u^{\xi}](t,z_2)\Big\} d z_1 d z_2.
\end{align*}
Therefore, if we choose 
\begin{equation*}
\xi=\xi(t,s,R)=\left\{
\begin{array}{rr}
             -\frac{\int \mathrm{Im}(\chi_R^2(x_1-s)\bar{u}(t,z_1)\nabla_x u(t,z_1))\,dz_1}{\int\chi_R^2(x_1-s)|u(t,z_1)|^2\,dz_1}, &\text{if }\int \chi_R^2(x_1-s)|u(t,z_1)|^2\,dz_1\neq 0,\\
             0, &\text{if }\int \chi_R^2(x_1-s)|u(t,z_1)|^2\,dz_1=0,
             \end{array}
\right.
\end{equation*}
then
\begin{equation}\label{aest5}
    \begin{aligned}
    \eqref{subest1}&+\eqref{subest3}\\&=\frac{4}{\omega_d R^d} 
    \iint\chi^2\left(\frac{x_1-s}{R}\right) \chi^2\left(\frac{x_2-s}{R}\right)|u(t,z_2)|^2\left|\nabla_x u^{\xi}(t,z_1)\right|^2 d z_1 d z_2 d s
\end{aligned}
\end{equation}
In view of estimates \eqref{aest1}-\eqref{aest5}, we arrive at
\begin{align}
&\frac{4}{\omega_d R^d} 
\iiint \chi^2\left(\frac{x_1-s}{R}\right) \chi^2\left(\frac{x_2-s}{R}\right)|u(t,z_2)|^2\left|\nabla u^{\xi}(t,z_1)\right|^2 d z_1 d z_2 d s\notag\\
& -\frac{2\alpha d}{(\alpha+2) \omega_d R^d} 
\iiint \chi^2\left(\frac{x_2-s}{R}\right) \chi^{\alpha+2}\left(\frac{x_1-s}{R}\right)|u(t,z_2)|^2|u(t,z_1)|^{\alpha+2} d z_1 d z_2 d s\notag\\
    &\leq \frac{d}{dt}M(t)+|\eqref{subest8}|+|\eqref{subest6}|+|\eqref{subest7}|.\label{interme1}
\end{align}
Next, we observe that 
\begin{align}
   &\hspace{3ex} \int|\nabla_x(\chi(\frac{x_1-s}{R}) u^{\xi}(t,z_1))|^2 d z_1\\&=\int\chi^2(\frac{x_1-s}{R})|\nabla_x u^{\xi}(t,z_1)|^2-\chi(\frac{x_1-s}{R}) \Delta_x \chi(\frac{x_2-s}{R})|u^{\xi}(t,z_1)|^2 d z_1\notag\\
    &= \int|\nabla_x(\chi(\frac{x_1-s}{R}) u^{\xi}(t,z_1))|^2 d z_1+O(\frac{M(u)}{R^2}).
\end{align}
Therefore, by  Lemma \ref{lem coercivity}, \eqref{interme1} further implies that
\begin{align}
&\frac{1}{R^d}\iiint  \chi^2\left(\frac{x_2-s}{R}\right)|u(t,z_2)|^2\left|\nabla_x(\chi(\frac{x_1-s}{R})u^{\xi}(t,z_1))\right|^2 d z_1 d z_2 d s\notag\\
   & \lesssim \frac{1}{R^d}\iiint  \chi^2\left(\frac{x_2-s}{R}\right)|u(t,z_2)|^2Q(\chi(\frac{x_1-s}{R})u^{\xi}(t,z_1))d z_1 d z_2 d s\notag\\
   &\lesssim \text{LHS of }\eqref{interme1}+O(\frac{M(u)}{R^2}) \notag\\
   &\lesssim\frac{d}{dt}M(t)+|\eqref{subest8}|+|\eqref{subest6}|+|\eqref{subest7}|+O(\frac{M(u)}{R^2}).\label{bbz}
\end{align}
Recalling Lemma \ref{global2}, we have $\|u(t)\|_{\Sigma}\lesssim m_w$ uniformly for $t\in \R$. Therefore, by H\"older's inequality and Sobolev embedding,
\begin{equation}\label{Mbound}
    \sup_{t\in \R}|M(t)|\lesssim  m_{\omega}^2R,
\end{equation}
which together with the fundamental theorem of calculus, implies that
\begin{align}\label{finala}
    \frac{1}{T_0} \int_I \frac{1}{J_0} \int_{R_0}^{e^{J_0} R_0}\frac{d}{dt}M(t) \frac{d R}{R} d t &\lesssim \frac{m_w^2}{J_0T_0} \int_{R_0}^{e^{J_0} R_0} d R\notag\\
&\lesssim \frac{m_w^2e^{{J_0}}R_0}{{J_0}T_0} .
\end{align}
Moreover, since 
\[
|\phi_R(x)-\psi_R(x)| \lesssim \min \left\{\frac{|x|}{\eta R}, \frac{R}{\eta |x|}\right\} ,\,\,|\nabla \phi_R| \lesssim \frac{1}{\eta R}, \,\, \text { and } |\nabla_x [\phi_R-\psi_R]| \lesssim \min \left\{\frac{1}{\eta R}, \frac{R}{\eta |x|^2}\right\},
\]
we consequently obtain that
\begin{align}\label{finala2}
    \frac{1}{T_0} \int_I \frac{1}{{J_0}} \int_{R_0}^{e^{J_0} R_0}|\eqref{subest8}| \frac{d R}{R} d t &\lesssim \frac{m_w^2}{{\eta J_0}} \int_{R_0}^{e^{J_0} R_0} \frac{d R}{R^2}\lesssim \frac{m_w^2}{{\eta J_0} R_0} 
\end{align}
and
\begin{align}\label{finala3}
    \frac{1}{T_0} \int_I \frac{1}{{J_0}} \int_{R_0}^{e^{J_0} R_0}|\eqref{subest6}| \frac{d R}{R} d t &\lesssim \frac{m_w^2}{{\eta J_0}} \max_{x_1,x_2\in\R^d}\int_{R_0}^{e^{J_0} R_0} \min \left\{\frac{|x_1-x_2|}{R}, \frac{R}{|x_1-x_2|}\right\}\frac{d R}{R}\notag\\
&\lesssim \frac{m_w^2}{{\eta J_0}}.
\end{align}
Similarly, as one may easily check that
\[
\left|\phi_R(x)-\varphi_R(x)\right| \lesssim \eta,
\] 
we also have
\begin{align}\label{finalb}
    \frac{1}{T_0} \int_I \frac{1}{{J_0}} \int_{R_0}^{e^{J_0} R_0}|\eqref{subest6}| \frac{d R}{R} d t &\lesssim \frac{m_w^2\eta}{{J_0}} \int_{R_0}^{e^{J_0} R_0} \frac{d R}{R}\lesssim m_w^2\eta.
\end{align}
Inserting \eqref{finala}-\eqref{finalb} into \eqref{bbz}, we see that
\begin{align}
    \text{LHS of } \eqref{thm5.1e}\lesssim m_{\omega}^2[\frac{e^{J_0}R_0}{{J_0}T_0}+\frac{e^{J_0}R_0}{{\eta J_0}T_0}+\frac{1}{{\eta J_0}}+\frac{1}{{\eta J_0}R_0}+\eta],  
\end{align}
which clearly implies Theorem \ref{thm5.1} by taking $\eta=\varepsilon$, $J_0=\varepsilon^{-3}$, $R_0=\varepsilon^{-1}$ and $T_0=e^{\varepsilon^{-3}}$. 
\end{proof}
\section{Proof of the main results}
This section is devoted to the proof of Theorem \ref{thm:alldimension}. We begin with the following results.
\begin{theorem}\label{thm:main}
    Let $u$ be a solution of \eqref{PHNLS} with   initial data     
    $u_0\in P_{\omega}^+$.  Then the solution $u(t)$ is globally well-posed and scatters in the sense that there exists $\varphi_\pm\in \Sigma $ such that 
       \begin{align*}
                  \lim_{{t\to \pm\infty}}    \big\|u-e^{it(\Delta_{z}-y^2)}\varphi_\pm\big\|_{\Sigma }=0. 
       \end{align*}
       \end{theorem}
  \begin{proof}
      For simplicity, we focus on the positive time direction. The negative case can be treated in a similar way.  In view of the scattering criterion (Lemma \ref{scattering-cri}), to prove scattering, it suffices to verify that for any $\varepsilon>0$, there exists  $T_0(\varepsilon,u)\gg1$ such that for any $a\in\R$,
 there exists $T\in(a,a+T_0)$ such that $[T-\varepsilon^{-\sigma},T]\in(a,a+T_0)$, there holds
 \begin{align}\label{mlm}
    \big\|u\big\|_{L_t^{q}L_x^r\mathcal{H}_y^s([T-\varepsilon^{-\sigma},T]\times\R^{d+1})}\lesssim\varepsilon^\mu
\end{align}
for some $\sigma,\mu>0$. From the interaction Morawetz estimate, there exists $T_0:=T_0(\varepsilon,u)\gg1$, ${J_0}={J_0}(\varepsilon,u)\gg1$, $R_0=R_0(\varepsilon,u_0)\gg1$ and $\eta_\varepsilon\ll1$ such that
\begin{align*}
    &\frac{1}{{J_0}T_0}\int_a^{a+T_0}\int_{R_0}^{R_0e^{J_0}}\frac{1}{R^d}\iiint \big|\chi_R(x_2-s)u(t,z_2)\big|^2\times\\
    &\hspace{20ex}\times\big|\nabla_x(\chi_R(x_1-s)u^\xi(t,z_1))\big|^2\,dz_1\,dz_2\,ds\frac{dR}{R}\,dt\lesssim\varepsilon.
\end{align*}
Then by H\"older's ineqaulity, we have
\begin{align*}
    \frac{1}{T_0}\int_a^{a+T_0}\frac{1}{R_1^d}\int\big\|\chi_{R_1}(\cdot-z)u(t)\big\|_{L_{x,y}^2}^2\big\|\nabla_x(\chi_{R_1}(\cdot-z)u^\xi(t))\big\|_{L_{x,y}^2}^2\,dz\,dt\lesssim\varepsilon.
\end{align*}
Denote by $z=\frac{R_1}{4}(w+\theta)$ with $w\in\Z^d$ and $\theta\in[0,1]^d$. For $\theta_0:\Z^d\to[0,1]^d$, it follows from the basic mean value theorem,
\begin{align*}
\frac{1}{T_0}\int_{a}^{a+T_0}\sum_{w\in\Z^d}\Big\|\chi_{R_1}\big(\cdot-\frac{R_1}{4}(w+\theta_0)\big)u(t)\Big\|_{L_{x,y}^2}^2\Big\|\nabla_x(\chi_{R_1}\big(\cdot-\frac{R_1}{4}(w+\theta_0)\big)u^\xi(t))\Big\|_{L_{x,y}^2}^2\,dt\lesssim\varepsilon.
\end{align*}
Spliting the time interval $\big[a+\frac{T_0}{2},a+\frac{3T_0}{4}\big]$ into $T_0\varepsilon^{\sigma}$ pieces of sub-intervals with length $\varepsilon^{-\sigma}$.  We claim that there exists at least one $T\in[a+\frac{T_0}{2},a+\frac{3T_0}{4}]$ such that $[T-\varepsilon^{-\sigma},T]\subset[a,a+T_0]$. Consequently, we have
\begin{align*}
\int_{T}^{T-\varepsilon^{-\sigma}}\sum_{w\in\Z^d}\Big\|\chi_{R_1}\big(\cdot-\frac{R_1}{4}(w+\theta_0)\big)u(t)\Big\|_{L_{x,y}^2}^2\Big\|\nabla_x(\chi_{R_1}\big(\cdot-\frac{R_1}{4}(w+\theta_0)\big)u^\xi(t))\Big\|_{L_{x,y}^2}^2\,dt\lesssim\varepsilon^{1-\sigma}.
\end{align*}
Combining with the refined Gagliardo-Nirenberg inequality,
\begin{align*}
    \big\|u\|_{L_x^{\frac{2d}{d-1}}}^4\lesssim\|u\|_{L_x^2}^2\big\|\nabla_xu^\xi\big\|_{L_x^2}^2,
\end{align*}
one has
\begin{align}
\int_{T-\varepsilon^{-\sigma}}^{T}\sum_{w\in\Z^d}\Big\|\chi_{R_1}\Big(\cdot-\frac{R_1}{4}(w+\theta_0)\Big)u(t)\Big\|_{L_x^\frac{2d}{d-1}L_y^2}^4\,dt\lesssim\varepsilon^{1-\sigma}.\label{L3-est1}
\end{align}
From the Sobolev embedding, H\"older's inequality, we have
\begin{align*}
    &\hspace{3ex}\sum_{w\in\Z^d}\Big\|\chi_{R_1}\Big(\cdot-\frac{R_1}{4}(w+\theta_0)\Big)u(t)\Big\|_{L_x^\frac{2d}{d-1}L_y^2}^2\\
    &\lesssim\sum_{w\in\Z^d}\Big\|\chi_{R_1}\Big(\cdot-\frac{R_1}{4}(w+\theta_0)\Big)u(t)\Big\|_{L_{x,y}^2}\Big\|\chi_{R_1}\Big(\cdot-\frac{R_1}{4}(w+\theta_0)\Big)u(t)\Big\|_{L_{x}^\frac{2d}{d-2}L_y^2}\\
    &\lesssim\Big(\sum_{w\in\Z^d}\Big\|\chi_{R_1}\Big(\cdot-\frac{R_1}{4}(w+\theta_0)\Big)u(t)\Big\|_{L_{x,y}^2}^2\Big)^\frac12\Big(\sum_{w\in\Z^d}\Big\|\chi_{R_1}\Big(\cdot-\frac{R_1}{4}(w+\theta_0)\Big)u(t)\Big\|_{L_{x}^\frac{2d}{d-2}L_y^2}^2\Big)^\frac12\\
    &\lesssim\|u(t)\|_{L_{x,y}^2}\Big(\sum_{w\in\Z^d}\Big\|\chi_{R_1}\Big(\cdot-\frac{R_1}{4}(w+\theta_0)\Big)u(t)\Big\|_{L_{x}^\frac{2d}{d-2}L_y^2}^2\Big)^\frac12.
    \end{align*}
    By the Sobolev inequality and the basic property $|\nabla\chi|\lesssim\eta^{-1}$, for $R_1>\eta^{-1}$, we have
    \begin{align}
&\hspace{3ex}\sum_{w\in\Z^d}\Big\|\chi_{R_1}\Big(\cdot-\frac{R_1}{4}(w+\theta_0)\Big)u(t)\Big\|_{L_{x}^\frac{2d}{d-2}L_y^2}^2\notag\\
&\lesssim\sum_{w\in\Z^d}\Big\|\chi_{R_1}\Big(\cdot-\frac{R_1}{4}(w+\theta_0)\Big)\nabla_xu(t)\Big\|_{L_{x,y}^2}^2+\frac{1}{R_1^2}\Big\|(\nabla\chi)_{R_1}\Big(\cdot-\frac{R_1}{4}(w+\theta_0)\Big)u(t)\Big\|_{L_{x,y}^2}^2\notag\\
&\lesssim\|u(t)\|_{\Sigma }^2+O(\eta^{-2}R_1^{-2})\|u(t)\|_{L_{x,y}^2}^2\lesssim\|u(t)\|_{\Sigma }^2.\label{L6-est}
    \end{align}
    Putting these estimates together, we have
    \begin{align*}
        \sum_{w\in\Z^d}\Big\|\chi_{R_1}\Big(\cdot-\frac{R_1}{4}(w+\theta_0)\Big)u(t)\Big\|_{L_x^\frac{2d}{d-1}L_y^2}^2\lesssim1.
    \end{align*}
    Taking the time integral, we have the following estimate directly,
    \begin{equation}
    \int_{T-\varepsilon^{-\sigma}}^T\sum_{w\in\Z^d}\Big\|\chi_{R_1}\Big(\cdot-\frac{R_1}{4}(w+\theta_0)\Big)u(t)\Big\|_{L_x^\frac{2d}{d-1}L_y^2}^2\lesssim\varepsilon^{-\sigma}.\label{L3-est2}
    \end{equation}
    Interpolating between \eqref{L3-est1} and \eqref{L3-est2} implies 
    \begin{align*}
      &\hspace{6ex}  \|u\|_{L_{t,x}^\frac{2d}{d-1}L_y^2([T-\varepsilon^{-\sigma},T]\times\R^{d+1})}^\frac{2d}{d-1}\\&\lesssim\int_{T-\varepsilon^{-\sigma}}^T\sum_{w\in\Z^d}\Big\|\chi_{R_1}\Big(\cdot-\frac{R_1}{4}(w+\theta_0)\Big)u(t)\Big\|_{L_{x}^\frac{2d}{d-2}L_y^2}^\frac{2d}{d-1}\,dt\\
        &\lesssim\int_{T-\varepsilon^{-\sigma}}^T
\Big(\sum_{w\in\Z^d}
\|\chi_{R_1}\Big(\cdot-\frac{R_1}{4}(w+\theta_0))u(t)\Big\|^4_{L_x^\frac{2d}{d-1}L_y^2}\Big)^\frac{1}{d-1}
\notag\\
&\hspace{2ex}\times\Big(\sum_{w\in\Z^d}
\|\chi_{R_1}\Big(\cdot-\frac{R_1}{4}(w+\theta_0)\Big)u(t)\Big\|_{L_x^\frac{2d}{d-1}L_y^2}^2\Big)^{\frac{d-2}{d-1}}\,dt
\notag\\
&\lesssim
\Big(\int_{T-\varepsilon^{-\sigma}}^{T}\sum_{w\in\Z^d}
\|\chi_{R_1}\Big(\cdot-\frac{R_1}{4}(w+\theta_0)\Big)u(t)\Big\|_{L_x^{\frac{2d}{d-1}}L_y^2}^4\,dt\Big)^{\frac{1}{d-1}}
\notag\\
&\hspace{2ex}\times\Big(\int_{T-\varepsilon^{-\sigma}}^{T}\sum_{w\in\Z^d}
\|\chi_{R_1}\Big(\cdot-\frac{R_1}{4}(w+\theta_0)\Big)u(t)\Big\|_{L_x^{\frac{2d}{d-1}}L_y^2}^2\,dt\Big)^\frac{d-2}{d-1}
\lesssim \varepsilon^{\frac{1}{d-1}-\sigma}.
    \end{align*}
    In conclusion, we have
    \begin{align*}
        \|u\|_{L_{t,x}^{\frac{2d}{d-1}}L_y^2}\lesssim\varepsilon^{\frac{d-1}{2d}\big(\frac{1}{d-1}-\sigma\big)}.
    \end{align*}
    To finish the proof, we need to choose the parameters carefully. Let $\theta\in(0,1)$ to be determined later. Suppose that $(q_2,r_2)\in\R^2$ such that
    \begin{align*}
        \frac{1}{q}=\frac{(d-1)\theta}{2d}+\frac{1-\theta}{q_2},\,\,\frac{1}{r}=\frac{\theta(d-1)}{2d}+\frac{1-\theta}{r_2}.
    \end{align*}
    Taking $r^*\geq2,s>0$ such that 
    \begin{align*}
        \frac{1}{r_2}=\frac{1}{r^*}-\frac{s}{d},\,\,\frac{2}{q_2}+\frac{d}{r^*}=\frac{2}{\al}.
    \end{align*}
    Since $(q,r)$ is $H^{s_c}$-admissible, then we have the following relationship
    \begin{align*}
        \frac{2}{\al}=\theta(d-1)\big(\frac{1}{2}+\frac{1}{d}\big)+(1-\theta)\big(\frac{d}{2}-s\big).
    \end{align*}
    Equivalently, we can write
    \begin{align}
        s=\frac{d}{2}-\frac{1}{1-\theta}\Big(\frac{2}{\al}-\frac{(d-1)(d+2)\theta}{2d}\Big).
    \end{align}
    Moreover, we need that $s\in(s_c,\frac{1}{2})$. By simple calculation, $s>s_c$ unconditionally.  $s<\frac{1}{2}$ is equivalent to 
    \begin{align*}
        \frac{d-1}{2}-\frac{2}{\al(1-\theta)}+\frac{(d-1)(d+2)\theta}{2d(1-\theta)}<0.
    \end{align*}
    Moreover, since $\al<\frac{4}{d-1}$, we have
    \begin{align*}
    \theta<\frac{d}{d-1}\big(\frac{2}{\al}-\frac{2}{d-1}\big).
    \end{align*}
    By H\"older inequality and Sobolev embedding, one has
    \begin{align*}
    \|u\|_{L_t^qL_x^r\mathcal H_y^s([T-\varepsilon^{-\si},T]\times\R^{d+1})}&\lesssim\|u\|_{L_{t,x}^\frac{2d}{d-1}L_y^2}^\theta\|u\|_{L_t^{q_2}L_x^{r_2}\mathcal H_y^{1-s}}^{1-\theta}\\
    &\lesssim\|u\|_{L_{t,x}^\frac{2d}{d-1}L_y^2}^\theta\|u\|_{L_t^{q_2}W_x^{s,r^*}\mathcal H_y^{1-s}}^{1-\theta}\\
    &\lesssim\varepsilon^{\frac{(d-1)\theta}{2d}\big(\frac{1}{d-1}-\sigma\big)}\varepsilon^{-\frac{\sigma}{q_2}(1-\theta)}\lesssim\varepsilon^{\frac{\theta}{2d}-\sigma(\frac{d-1}{2d}\theta+\frac{1-\theta}{q_2})}.
    \end{align*}
    Choosing $\sigma\ll1$ small enough, we can complete the proof of Theorem \ref{thm:main}.
  \end{proof}
  

\begin{lemma}\label{lemma msequaltomgs}
    We have $m_{\omega}=m_{GS}$. Hence, $ P_{\omega}^+$=$ \mathcal{K}_{\omega}^+$.
\end{lemma}
\begin{proof}
    By Theorem \ref{thm groundstate} and the definition of $m_{\omega}$ and $m_{GS}$, we know that
    \[
    m_{GS}\leq m_{\omega}.
    \]
To prove the reserve inequality,  we first recall    that Ardila and Carles \cite{Ardila-Carles2021} proved that $m_{GS}$ can be attained by a non-trivial ground state $\phi_{GS}\in \Sigma $, i.e., $S_{\omega}(\phi_{GS})=m_{GS}$. Therefore,  it suffices to show that
    \[
    Q(\phi_{GS})=0.
    \]
 Recalling the definition of $\phi_{GS}$, we know that $e^{i\omega t}\phi_{GS}$ is a solution of \eqref{PHNLS}. Let $\tilde \varphi$ be a radial smooth solution in $C_0^{\infty}(\R^d)$ such that
    \[
    \tilde \varphi(r)=r^2, \quad r=|x|\leq 1,\quad \tilde \varphi(r)\geq0\quad\text{and}\quad \tilde \varphi(r)''\leq 2,\quad r\geq 0.
    \]
    and let $\tilde \varphi_{R}=R^2\tilde \varphi(\frac{x}{R})$. Then for any $R>0$, we know that the following truncation semi-virial functional
    \[
    V_{R}(t):=\int_{\R^{d+1}}\tilde \varphi_R(x)|e^{i\omega t}\phi_{GS}(z)|^2 dz
    \] 
    is well-defined and is equal to  a constant independent of $t
    \in\R$. Therefore, $V_{R}^{''}(0)\equiv0$ for all $R>0$. On the other hand, by a direct computation, we see that
    \begin{align}
    V_{R}^{''}(0)&=8Q(\phi_{GS})+4\int_{\R^{d+1}}\Big(\frac{\tilde \varphi_R^{\prime\prime}}{r}-2\Big)|\nabla_x\phi_{GS}|^2\,dz+4\int_{\R^{d+1}}\Big(\frac{\tilde \varphi_R^{\prime\prime}}{r^2}-\frac{\tilde \varphi_R^\prime}{r^3}\Big)|x\cdot\nabla_x\phi_{GS}|^2\,dz\notag\\
    &-\frac{2\al}{\al+2}\int_{\R^{d+1}}\big(\tilde \varphi_R^{\prime\prime}+(d-1)\frac{\tilde \varphi_R^\prime}{r}-2d\big)|\phi_{GS}|^{\al+2}\,dz-\int_{\R^{d+1}}\Delta_x^2\tilde \varphi_R |\phi_{GS}|^2\,dz\notag\\    &=8Q(\phi_{GS})+A_R.\label{trunv}
    \end{align}
     It is easy to verify that 
     \begin{align*}
         \tilde \varphi_R(r)\leq r^2,\,\tilde \varphi_R^{\prime\prime}(r)\leq2,\,|\Delta_x^2\tilde \varphi_R(r)|\leq CR^{-2}.
     \end{align*}
    Using the support condition of $\tilde \varphi_R$, we see that 
    \begin{align*}
        \operatorname{supp}(\tilde \varphi_R)=\operatorname{supp}(|\tilde \varphi_R^{\prime\prime}+(d-1)\tilde \varphi_R^\prime/r-2d|)\subset \{R\leq|x|<\infty\},
    \end{align*}
     then we can prove that 
     \begin{align*}
         -\frac{2\al}{\al+2}\int_{\R^{d+1}}(\tilde \varphi_R^{\prime\prime}+(d-1)\tilde \varphi_R^\prime/r-2d)|\phi_{GS}|^{\al+2}\,dz\lesssim\|\phi_{GS}\|_{L^{\al+2}(|x|\geq R)}^{\al+2}.
     \end{align*}
     Also, we can prove that 
     \begin{align*}
         \int_{\R^{d+1}}\Delta_x^2\tilde \varphi_R|\phi_{GS}|^2\,dz\lesssim R ^{-2}\|\phi_{GS}\|_{L^2(|x|\geq R)},
     \end{align*}
     \begin{align*}
4\int_{\R^{d+1}}\Big(\frac{\tilde \varphi^{\prime\prime}}{r}-2\Big)|\nabla_x\phi_{GS}|^2\,dz\lesssim\|\nabla_x\phi_{GS}\|_{L^2(|x|\geq R)}^2,
     \end{align*}
     \begin{align*}   4\int_{\R^{d+1}}\Big(\frac{\tilde \varphi_R^{\prime\prime}}{r^2}-\frac{\tilde \varphi_R^\prime}{r^3}\Big)|x\cdot\nabla_x\phi_{GS}|^2\,dz\lesssim\|\nabla_x\phi_{GS}\|_{L^2(|x|\geq R)}^2.
     \end{align*}      
     Since $\phi_{GS}\in\Sigma $, the right-hand side of the above estimates are well-defined. Then by monotone convergence theorem, we have
    $|A_R|\to0$ as $R\to\infty$, which, in view of \eqref{trunv}, implies $Q(\phi_{GS})=0$.
\end{proof}

\begin{proof}[Proof of Theorem \ref{thm:alldimension}] The proof of the scattering result  is  a combination of  Theorem \ref{thm:main} and Lemma \ref{lemma msequaltomgs}. The proof of the blow-up result is proved via Glassey's convex method \cite{Glassey}   and can be obtained by following  the arguments in  \cite[Section 4]{Ardila-Carles2021}.  
\end{proof}

{\small \noindent \textbf{Acknowledgements:}
T. Liu is supported by the   National Funded Postdoctoral Researcher Program    (GZB20240945) and China Postdoctoral Science Foundation  (2025M784442).
J. Zheng was supported by  National key R\&D program of China: 2021YFA1002500 and  NSFC Grant 12271051.}
				
				\medskip
%
%
%
%

\end{document}